\documentclass[a4paper, 11pt]{amsart}

\usepackage{amsfonts, amsmath, amssymb,amsthm,bbm,comment,fancyhdr,mathrsfs,verbatim,mathabx}

\usepackage{a4wide}

\newcommand{\seq}{\subseteq}

\newcommand{\vs}{\vspace*}

\newcommand{\nin}{\noindent}
\newtheorem{mthm}{Theorem}[section]
\newtheorem{mylem}[mthm]{Lemma}
\newtheorem{myprn}[mthm]{Proposition}
\newtheorem{mycor}[mthm]{Corollary}
\newtheorem{mydef}[mthm]{Definition}
\newtheorem{myrem}[mthm]{Remark}
\newtheorem{mycon}[mthm]{Construction}
\newtheorem{myeg} [mthm]{Example}
\newtheorem{myque} [mthm]{Question}
\newtheorem{myalg} [mthm]{Algorithm}

\newenvironment{lem}{\begin{mylem}}{\end{mylem}}

\newenvironment{rem}{\begin{myrem}\rm}{\end{myrem}}

\def \nin {\noindent}

\def \Lemma #1 {\vs{3mm}\nin {\bf Lemma #1} \it}
\def \Prop #1 {\vs{3mm}\nin {\bf Proposition #1} \it}
\def \Th #1 {\vs{3mm}\nin {\bf Theorem #1} \it}
\def \Cor #1 {\vs{3mm}\nin {\bf Corollary #1} \it}
\def \Ex #1 {\vs{3mm}\nin {\bf Example #1} \it}

\def \part #1 {\hfil\break\hglue 12pt {\rm (#1)~}}

\usepackage{amsbsy, amsmath, amsfonts, braket, latexsym, amsthm, amsxtra, amscd, graphics, mathrsfs, verbatim}

\usepackage{amssymb}

\usepackage[usenames,dvipsnames]{xcolor}

\usepackage[shortlabels,inline]{enumitem}
\SetEnumerateShortLabel{a}{\textup{(\alph*)}}
\SetEnumerateShortLabel{A}{\textup{(\Alph*)}}
\SetEnumerateShortLabel{1}{\textup{(\arabic*)}}
\SetEnumerateShortLabel{i}{\textup{(\roman*)}}
\SetEnumerateShortLabel{I}{\textup{(\Roman*)}}

\usepackage[mathscr]{euscript}

\usepackage[pdfborder={0 0 0}]{hyperref} 
\hypersetup{breaklinks, raiselinks}

\usepackage[initials,lite,alphabetic]{amsrefs} 

\usepackage[all]{xy}

\theoremstyle{plain}

\newtheorem{theorem}[mthm]{Theorem}

\newtheorem{lemma}[mthm]{Lemma}

\newtheorem*{proper*}{Property}
\newtheorem{proposition}[mthm]{Proposition}

\newtheorem{corollary}[mthm]{Corollary}

\newtheorem*{lm*}{Lemma}
\newtheorem*{mthm*}{Theorem}

\theoremstyle{definition}

\newtheorem{definition}[mthm]{Definition}

\newtheorem*{df*}{Definition}

\newtheorem{dis}[mthm]{Discussion}

\newtheorem{example}[mthm]{Example}

\newtheorem{ex-notn}[mthm]{Example/Notation}

\theoremstyle{remark}

\newtheorem{remark}[mthm]{Remark}

\newtheorem*{acknowledgement*}{Acknowledgement}

\newtheorem*{ex*}{Example}
\newtheorem*{exer*}{Exercise}
\newtheorem*{rem*}{Remark}
\newtheorem*{prob*}{Problem}
\newtheorem*{prop*}{Proposition}

\def\diam{\operatorname{diam}}

\def\g{\operatorname{g}}

\def\Ini{\operatorname{Ini}} 

\def\Inv{\operatorname{Inv}}

\def\link{\operatorname{link}}

\def\pure{\operatorname{pure}}

\def\KK{{\mathbb K}}

\def\NN{{\mathbb N}}

\def\ZZ{{\mathbb Z}}

\def\calE{\mathcal{E}}
\def\calF{\mathcal{F}}

\def\calV{\mathcal{V}}

\usepackage{bm}

\def\bdc{{\bm c}}

\def\bdd{{\bm d}}

\def\bdh{{\bm h}}

\def\bdm{{\bm m}}

\def\bdx{{\bm x}}

\def\alert#1{\textcolor{Magenta}{#1}}
\def\bar#1{\overline{#1}}

\def\dis{\operatorname{dis}}

\def\Index#1{\emph{#1}}

\def\coloneqq{\mathrel{\mathop:}=}

\def\sqr#1#2{{\vcenter{\hrule height.#2pt
        \hbox{\vrule width.#2pt height#1pt \kern#1pt
                \vrule width.#2pt}
        \hrule height.#2pt}}}

\def\opn#1#2{\def#1{\operatorname{#2}}} 
\opn\rev{rev}
\opn\Lex{Lex}
\opn\GL{GL}
\opn\initial{in}

\usepackage{floatrow}
\usepackage{caption}
\DeclareCaptionSubType[alph]{figure}
\begin{document}
\title{Strong shellability of simplicial complexes}

\author{Jin Guo}
\address{College of Information Science and Technology, Hainan University, Haikou, Hainan, 570228, P.R. China}
\email{guojinecho@163.com}

\author{Yi-Huang Shen}
\address{Wu Wen-Tsun Key Laboratory of Mathematics of CAS and School of Mathematical Sciences, University of Science and Technology of China, Hefei, Anhui, 230026, P.R. China}
\email{yhshen@ustc.edu.cn}

\author{Tongsuo Wu}
\address{Department of Mathematics, Shanghai Jiaotong University, Minhang, Shanghai, 200240, P.R. China}
\email{tswu@sjtu.edu.cn}


\begin{abstract}
    Imposing a strong condition on the linear order of shellable complexes, we introduce strong shellability. Basic properties, including the existence of dimension-decreasing strong shelling orders, are developed with respect to  nonpure strongly shellable complexes. Meanwhile, pure strongly shellable complexes can be characterized by the corresponding codimension one graphs. In addition, we show that the facet ideals of pure strongly shellable complexes have linear quotients.
\end{abstract}

\thanks{This research is supported by the National Natural Science Foundation of China (Grant No.~11201445, 11526065 and 11271250) and  the Fundamental Research Funds for the Central Universities. }



\maketitle

\section{Introduction }

Recall that a \Index{simplicial complex} $\Delta$ on a vertex set $\calV(\Delta)=V$ is a finite subset of $2^{V}$, such that $A \in \Delta$ and $B \seq A$ implies $B \in \Delta$. The set $A$ is called a \Index{face} if $A \in \Delta$, and called a \Index{facet} if $A$ is a maximal face with respect to inclusion.
The set of facets of $\Delta$ will be denoted by $\calF(\Delta)$. Typically, for the complexes considered here, the vertex set is the set $[n]\coloneqq\Set{1,2,\dots,n}$ for some $n\in \ZZ_+$.

When $\calF(\Delta)=\Set{F_1, \cdots, F_m}$, we write $\Delta=\braket{ F_1, \cdots, F_m }$.
In particular, if $\Delta$ has a unique facet $F$, then $\Delta$ is called a \Index{simplex}.
The \Index{dimension} of a face $A$, denoted by $\dim(A)$, is $|A|-1$.  The \Index{dimension} of a complex $\Delta$, denoted by $\dim(\Delta)$, is the maximum dimension of its faces. The complex $\Delta$ is called \Index{pure} if all the facets of $\Delta$ have the same dimension; otherwise, it will be called \Index{nonpure}.

Recall that a $d$-dimensional pure complex $\Delta$ on $[n]$ is called \Index{shellable} if  there exists a shelling order on its facet set $\mathcal{F}(\Delta)$, say $F_1, \ldots, F_m$, such that the subcomplex $\braket{F_1,\dots,F_{k-1}}\cap \braket{F_k}$ is pure of dimension $d-1$ for each $2\leq k \leq m$.
As pointed out in \cite{MR1453579},  ``Shellability  is a simple but powerful tool for proving the Cohen-Macaulay property, and almost all Cohen-Macaulay complexes arising `in nature' turn out to be shellable. Moreover, a number of invariants associated with Cohen-Macaulay complexes can be described more explicitly or computed more easily in the shellable case.'' Shellability was later generalized to nonpure complexes by Bj\"orner and Wachs \cite{MR1333388,  MR1401765}. Matroid complexes, shifted complexes and vertex decomposable complexes are all known to be important shellable complexes.

In the present paper,
we will impose a stronger requirement on the linear order of shellability, hence end up with
a new kind of complex, which will be called as \Index{strongly shellable}.
Its basic properties will be discussed in Section 2.
Using restriction maps,  Bj\"orner and Wachs \cite{MR1333388} showed that any nonpure shellable complex has a dimension-decreasing shelling order.
Using a different technique, we will establish a similar result for strongly shellable complexes; see Theorem \ref{thm:Rearrangement} and its corollary.

In Section 3, we introduce the codimension one graph of complexes.  We will investigate the strong shellability from this point of view. In Section 4, we focus on pure strongly shellable complexes and provide several equivalent characterizations, in particular, using its codimension one graph.
It is well-known that if $\Delta$ is a shellable complex, then the Stanley-Reisner ideal of its dual complex $\Delta^{\vee}$ has linear quotients. If in addition $\Delta$ is strongly shellable, we will show that
the facet ideal of $\Delta$ also has linear quotients. This property will be the main topic when we study the chordal (hyper)-graphs in the sequel paper \cite{ESSC}.

As applications, we will consider strongly shellable posets in Section 5.
In section 6, we will discuss the relationship between strongly shellable complexes and other shellable complexes in the pure case.  In the final section, we notice that strongly shellable complexes can be characterized by strong h-assignments, an easy generalization of a result of Moriyama \cite{ISI:000292480700004}.  This provides a relatively fast algorithm for checking strong shellability.

\section{Strongly shellable complexes, general case}

Shellability of general complexes were introduced by Bj\"orner and Wachs in \cite{MR1333388,MR1401765}.

\begin{definition}
    A complex $\Delta$ is \Index{shellable} if its facets can be ordered $F_1,F_2, \dots ,F_t$ such that the subcomplex $\braket{F_1,\dots,F_{k-1}}\cap \braket{F_k}$ is pure of dimension $\dim (F_k) - 1$ for all $k = 2, \dots , t$.
    Such an ordering of facets is called a \Index{shelling order} of $\Delta$.
\end{definition}

Sometimes, one may refer to shellable complexes as \Index{semipure shellable} or \Index{nonpure shellable} to emphasize that the complexes are not necessarily pure.  Related to this concept, we will impose additional requirement on the above linear order as follows.

\begin{definition}
    \label{nonpure strongly shellable}
    A complex $\Delta$ is called \Index{strongly shellable} if its facets can be ordered $F_1, F_2, \cdots, F_t$ such that for every $i$ and $j$ with $1\le i<j\le t$, there exists a $k$ with $1\le k<j$, such that:
    \begin{enumerate}[1]
        \item \label{nss-1} $|F_j \setminus F_k|=1$,
        \item \label{nss-2} $F_j \setminus F_k \seq F_j \setminus F_i$, and
        \item \label{nss-3} $F_k \setminus F_j \seq F_i$.
    \end{enumerate}
    Such an ordering of facets will be called a \Index{strong shelling order}  of $\Delta$, or more specifically, a \Index{strong shelling order} on $\calF(\Delta)$.
\end{definition}

Note that a complex $\Delta$ is shellable if and only if it satisfies the conditions \ref{nss-1} and \ref{nss-2} of Definition \ref{nonpure strongly shellable} by \cite[Lemma 2.3]{MR1333388}. Hence, for a simplicial complex, strong shellability is indeed stronger than shellability.

\begin{rem}
    \label{1/2/3 equivalent discription}
    It is easy to see that the conditions \ref{nss-1}, \ref{nss-2} and \ref{nss-3} of Definition \ref{nonpure strongly shellable} are equivalent to
    \begin{enumerate}
        \item[$(1)'$] $\dim(F_j) = \dim(F_j \cap F_k)+1$,
        \item[$(2)'$] $F_i \cap F_j \seq F_k$, and
        \item[$(3)'$] $F_k \seq F_i \cup F_j$
    \end{enumerate}
    respectively. In particular, we will always have $\dim(F_j)\le \dim(F_k)$. It will be helpful to think of the $F_i,F_j$ and $F_k$ in
    Definition \ref{nonpure strongly shellable} as depicted in Figure \ref{Fijk}, namely, we may assume that $F_i=P_1\sqcup P_2\sqcup P_3$, $F_k=P_2\sqcup P_3 \sqcup P_4$, $F_j=P_3\sqcup P_4 \sqcup P_5$ with $|P_5|=1$.
    Here, by $\sqcup$, we mean disjoint union. With this figure in mind,
     we will be able to see, for instance, that $|F_j\setminus F_i|=|F_k\setminus F_i|+1$.
    \begin{figure}[!ht]
        \begin{minipage}[h]{\linewidth} \centering
            \includegraphics{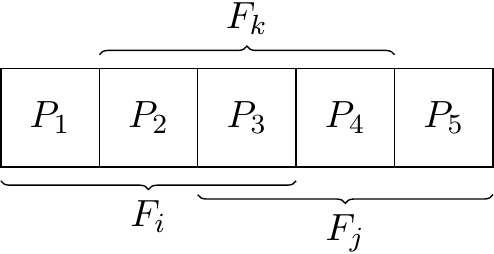}
        \end{minipage}
        \caption{Relations among $F_i$, $F_j$ and $F_k$} \label{Fijk}
    \end{figure}
\end{rem}

Taking advantage of the restriction map, Bj\"orner and Wachs \cite[Rearrangement lemma 2.6]{MR1333388} showed  that any  shelling order of a nonpure shellable complex can be rearranged to be a shelling order with decreasing dimensions. In the following, we will establish a similar result with respect to strong shellability.

\begin{lemma}
    \label{dim < dim}
    Let $\Delta$ be a strongly shellable complex. Then for each pair of facets $F_1$ and $F_2$ with $\dim(F_1) < \dim(F_2)$, there exists a facet $G$, such that $\dim(F_1)=\dim(F_1\cap G)+1$ and $F_1\cap F_2 \subseteq G \subseteq F_1\cup F_2$.
\end{lemma}

\begin{proof}
    Let $G_1, G_2, \cdots, G_t$ be a strong shelling order on $\calF(\Delta)$. Assume that $F_1 = G_{i_1}$ and $F_2 = G_{i_2}$.
    We claim that one can find suitable $G=G_{i}$, with $i<\max (i_1,i_2)$, satisfying the requirement. We establish this claim by induction on $t$. When $t=2$, this is easy. Thus, we may assume that $t\ge 3$.

    If $i_1 > i_2$, the existence of $G$ follows directly from Remark \ref{1/2/3 equivalent discription}. Thus, we may assume that $i_1 < i_2$.
    Since $G_1, G_2, \cdots, G_t$ is a strong shelling order, by Remark \ref{1/2/3 equivalent discription}, there exists some $i_3 < i_2$ such that $G_{i_1} \cap G_{i_2} \seq G_{i_3}\subseteq G_{i_1}\cup G_{i_2}$. Therefore,
    \begin{equation}
        G_{i_1} \cap G_{i_2} \seq G_{i_1} \cap G_{i_3}\quad \text{ and } \quad G_{i_1}\cup G_{i_3}\subseteq G_{i_1}\cup G_{i_2}. \label{i3-1}
    \end{equation}
    As $i_1,i_3<i_2\le t$ and $G_1,G_2,\dots,G_{i_2-1}$ forms a strong shelling order, by induction, we can find $G=G_i$ with $i<\max(i_1,i_3)$ such that
    \begin{equation}
        \dim(G_{i_1})=\dim(G_{i_1}\cap G)+1 \quad \text{ and }\quad G_{i_1}\cap G_{i_3} \subseteq G \subseteq G_{i_1}\cup G_{i_3}. \label{i3-2}
    \end{equation}
    Combining \eqref{i3-1} and \eqref{i3-2}, we see that the claim works.
\end{proof}

Let $\succ: F_1, \dots, F_t$ be a linear order on the facet set of $\Delta$, i.e., $F_i\succ F_j$ if and only $i<j$.  One can define an \Index{induced dimension-related order} $\vdash_{\succ}$ on $\mathcal{F}(\Delta)$ as follows: $F \vdash_{\succ} G$ if and only if
\begin{enumerate}[1]
    \item $\dim(F)>\dim(G)$, or
    \item $\dim(F)=\dim(G)$ and $F \succ G$.
\end{enumerate}
For a given $F\in \calF(\Delta)$, we will write
$\Ini_\succ(F)\coloneqq\Set{G\in \calF(\Delta)| G\succcurlyeq F}$ for the \Index{initial interval} with respect to $\succ$ and $F$.

The following result follows easily from Definition \ref{nonpure strongly shellable} and Remark \ref{1/2/3 equivalent discription}.

\begin{lemma}
    \label{lex and o are sso}
    Let $\Delta$ be a strongly shellable complex with a strong shelling order $\succ$. Then for each facet $F$, the restrictions of $\succ$ on both
    $\Ini_{\succ}(F)$ and $\Ini_{\vdash_\succ}(F)$
    are also strong shelling orders of the corresponding subcomplexes.
\end{lemma}

Let $\succ:F_1,\dots,F_t$ be a linear order on $\calF(\Delta)$.  Thus, the induced dimension-related order $\vdash_\succ$ is fixed.  If we are given another linear order $\succ'$ on $\calF(\Delta)$,  we will call $(F_i, F_j)$ a \Index{relative inverse pair with respect to $\succ$}, if $F_i \vdash_{\succ} F_j$ and $F_j \succ' F_i$. Denote by $\Inv_{\succ}(\succ')$ the set of relative inverse pairs with respect to $\succ$.
Obviously, one can recover the linear order $\succ'$ from $\vdash_\succ$ by switching the pairs of facets in $\Inv_{\succ}(\succ')$.

We will simply write $\Inv_\succ(\succ)$ as $\Inv(\succ)$. An ordered pair $(F_i, F_j) \in \Inv(\succ)$ precisely when $j<i$ and $\dim (F_j)<\dim(F_i)$.

\begin{lemma}
    Let $\succ$ and $\succ'$ be two linear orders on $\calF(\Delta)$. If $\Inv_{\succ}(\succ') \seq \Inv(\succ)$, then the induced dimension-related orders $\vdash_\succ$ and $\vdash_{\succ'}$ coincide.
    \label{vdashes}
\end{lemma}

\begin{proof}
    Let $F_i$ and $F_j$ be two distinct facets. When $\dim(F_i)>\dim(F_j)$, we have both $F_i\vdash_{\succ} F_j$ and $F_i\vdash_{\succ'}F_j$. Therefore, we may assume that $\dim(F_i)=\dim(F_j)$ and $F_i \succ F_j$. Whence, $F_i\vdash_\succ F_j$. Note that $(F_i,F_j)\notin \Inv(\succ)$. It follows from the condition $\Inv_{\succ}(\succ') \seq \Inv(\succ)$ that $(F_i,F_j)\notin \Inv_\succ(\succ')$. But $F_i\vdash_\succ F_j$. This simply means that $F_i \succ'F_j$ and in turn $F_i\vdash_{\succ'}F_j$.
\end{proof}

\begin{remark}
    \label{simplification}
    If $(F_i,F_j)\in \Inv(\succ)$, then $j<i$ with $\dim(F_j)<\dim(F_i)$. In this case, we will be able to find a $k$ with $j\le k<i$ such that $\dim(F_k)<\dim(F_{k+1})$. Now $(F_{k+1},F_{k})\in \Inv(\succ)$. We can modify $\succ$ by only switching the relation between $F_k$ and $F_{k+1}$ and end up with $\succ'$. Then, $\succ'$ is indeed a linear order with $\Inv_\succ(\succ')\subsetneq \Inv(\succ)$ and $\Inv(\succ)\setminus \Inv_\succ(\succ')=\{(F_{k+1},F_{k})\}$.
\end{remark}

\begin{theorem}
    [Rearrangement]
    \label{thm:Rearrangement}
    Let $\Delta$ be a nonpure strongly shellable complex with a strong shelling order $\succ$. Then any linear order $\succ'$ on $\mathcal{F}(\Delta)$ is also a strong shelling order, provided that $\Inv_{\succ}(\succ') \seq \Inv(\succ)$.
\end{theorem}

\begin{proof}
    Assume that $\succ:F_1,F_2,\dots,F_t$ gives the strong shelling order on $\calF(\Delta)$. Let $\succ'$ be another linear order with $\Inv_{\succ}(\succ') \seq \Inv(\succ)$. We may assume that this containment is strict, for otherwise, $\succ$ coincides with $\succ'$.
    By Lemma \ref{vdashes} and Remark \ref{simplification}, it suffices to consider the case when
    \[
    \Inv(\succ) \setminus \Inv_{\succ}(\succ') = \{(F_{i+1}, F_{i})\}.
    \]
    Whence, for arbitrary distinct $k_1,k_2\in[t]$, if $\Set{k_1,k_2}\ne \Set{i,i+1}$, then $F_{k_1}\succ F_{k_2}$ if and only if $F_{k_1}\succ' F_{k_2}$. On the other hand, $\dim(F_i)<\dim(F_{i+1})$ and $F_{i+1}\succ' F_i$.

    Note that the restriction
    \[
    \succ_{\Ini_{\succ}(F_{i+1})}:
    \, F_1, \cdots, F_{i-1}, F_i, F_{i+1}
    \]
    is a strong shelling order of $\Delta'\coloneqq\braket{F_1,\dots,F_{i+1}}$ by Lemma \ref{lex and o are sso}. We claim that the restriction
    \[
    (\succ')_{\Ini_{\succ'}(F_i)}:
    \, F_1, \cdots, F_{i-1}, F_{i+1}, F_i
    \]
    is still a strong shelling order of $\Delta'$. To this end, we only need to check the facets $F_{i+1}$ and $F_i$. Since $\dim(F_i)< \dim(F_{i+1})$, the claim follows from Lemma \ref{dim < dim}. Consequently, $\succ'$ is a strong shelling order of $\Delta$.
\end{proof}

\begin{corollary}
    \label{d dimension}
    Let $\Delta$ be a nonpure strongly shellable complex. Then, there exists a dimension-decreasing strong shelling order on $\mathcal{F}(\Delta)$.
\end{corollary}

\begin{proof}
    Let $\succ$ be an arbitrary strong shelling order on $\calF(\Delta)$  and choose $\vdash_\succ$ as $\succ'$. Then $\Inv_\succ(\succ')=\varnothing$. Thus, by Theorem \ref{thm:Rearrangement}, $\vdash_\succ$ is also a strong shelling order. On the other hand, $\vdash_\succ$ trivially satisfies the dimensional requirement.
\end{proof}

Let $\Delta$ be a complex and let $0\le i \le \dim(\Delta)$.
Recall that the \Index{$i$-th skeleton} of $\Delta$, denoted by $\Delta^{(i)}$, is the subcomplex of $\Delta$ generated by all faces of $\Delta$ of dimension at most $i$.
On the other hand, the \Index{pure $i$-th skeleton} of $\Delta$, denoted by $\Delta^{[i]}$, is the pure subcomplex of $\Delta$ generated by all $i$-dimensional faces.
The following result is clear.

\begin{proposition}
    \label{pure d}
    Let $\Delta$ be a strongly shellable  complex of dimension $d$. Then $\Delta^{[d]}$ is also strongly shellable.
\end{proposition}

It is shown in \cite[Theorem 8.2.18]{MR2724673} that for any shellable complex, its skeletons and pure skeletons are still shellable. This property is not preserved for strong shellability of complexes, as the following example shows.

\begin{example}
    Let $\Delta$ be a complex with the facet set
    \[
    \Set{\{1, 2, 3\},  \{2, 3, 4\},  \{3, 4, 5\},\{4, 5, 6\}}.
    \]
    It is direct to check that $\Delta$ is strongly shellable, but
    \[
    \Delta^{(1)}=\Delta^{[1]}=\braket{ \{1, 2\}, \{2, 3\},\{1, 3\}, \{2, 4\},\{3, 4\}, \{3, 5\}, \{4, 5\}, \{4, 6\}, \{5, 6\} }
    \]
    is not strongly shellable. Actually, using the terminology introduced in the next section, we have
    \[
    \dis_{\Delta}(\{1, 2\}, \{5, 6\})=2,
    \]
    and
    \[
    \dis_{\Gamma(\Delta)}(\{1, 2\}, \{5, 6\})=3.
    \]
    Thus, by Theorem \ref{distance of graph less than complex}, the skeleton $\Delta^{[1]}$ is not strongly shellable.
\end{example}

For a complex $\Delta$,
we denote by $\overline{\pure_k}(\Delta)$ the pure complex generated by the $k$-dimensional facets of $\Delta$.
Generally speaking, $\overline{\pure_k}(\Delta)$ may not necessarily be strongly shellable even though $\Delta$ is. For example, if $\Delta=\braket{ \{1245\}, \{123\}, \{456\} }$, it is easy to check that $\Delta$ is strongly shellable, while   $\overline{\pure_2}(\Delta)=\braket{\{123\}, \{456\}}$ is not. However, in the special case when $k=1$, the following result holds:

\begin{proposition}
    \label{pure_1 ss}
    Let $\Delta$ be a strongly shellable complex of positive dimension. Then $\overline{\pure_1}(\Delta)$ is also strongly shellable.
\end{proposition}

\begin{proof}
    Assume that $F_1, \cdots, F_t$ is a strong shelling order on $\calF(\Delta)$, with $\dim(F_i) \geq \dim(F_j)$ whenever $i<j$.  We may assume that $\overline{\pure_1}(\Delta) = \braket{ F_{s+1}, \cdots, F_t }$. For each pair $F_i$ and $F_j$ with $s+1 \leq i<j \leq t$, there exists a $k$ with $1\le k<t$, such that $F_k \subseteq F_i\cup F_j$. Notice that $\dim(F_k)\ge \dim(F_j)=\dim(F_i)=1$.  If $\dim(F_k)\ge 2$, this will force $F_i\subsetneq F_k$ or $F_j\subsetneq F_k$. Thus, we have indeed $\dim(F_k)=1$, and $F_k \in \overline{\pure_1}(\Delta)$. This shows that $F_{s+1}, \cdots, F_t$ is a strong shelling order of $\overline{\pure_1}(\Delta)$.
\end{proof}

The above result hints that one-dimensional pure strongly shellable complexes are very special. In our next paper \cite{ESSC}, we will show that they are indeed the complement graphs of chordal graphs.

Recall that for a simplicial complex $\Delta$, the \Index{link} of a face $A \in \Delta$ is defined as
\[
\link_{\Delta}(A)\coloneqq \Set{B \in \Delta \mid B \cup A \in \Delta, \, B \cap A =\varnothing}.
\]
When $A=\Set{x}$, we will simply write it as $\link_\Delta(x)$.

On the other hand, for a subset $W\subseteq \calV(\Delta)$, the \Index{restriction} of $\Delta$ on $W$ is the subcomplex
\[
\Delta_W\coloneqq \Set{F\in \Delta: F\subseteq W}.
\]
The restriction is sometimes denoted by $\Delta[W]$ as well. For a vertex $x\in \calV(\Delta)$, we will usually write $\Delta_{\calV(\Delta)\setminus x}$ as $\Delta\setminus x$.

\begin{proposition}
    \label{link ss}
    Let $\Delta$ be a strongly shellable complex. Then $\link_{\Delta}(A)$ is also strongly shellable for any $A \in \Delta$.
\end{proposition}

\begin{proof}
    Assume that $\succ$ is a strong shelling order on $\calF(\Delta)$. It induces a linear order $\succ'$ on $\calF(\link_\Delta(A))=\Set{F\setminus A\mid A\subseteq F\in \calF(\Delta)}$:
    \[
    \text{$F'\succ' G'$\quad if and only if \quad $(F'\sqcup A)\succ (G'\sqcup A)$ in $\calF(\Delta)$.}
\]
Now, take arbitrary pair $F_i'\succ' F_j'$ in $\calF(\link_\Delta(A))$. Equivalently, $(F_i'\sqcup A) \succ (F_j' \sqcup A)$ in $\calF(\Delta)$. Hence, there exists some $F_k\in \calF(\Delta)$ such that $F_k\succ (F_j'\sqcup A)$ and
    \[
    |(F_j'\sqcup A) \setminus F_k|=1 \quad \text{ and } \quad  (F_i'\sqcup A)\cap (F_j'\sqcup A) \subseteq F_k \subseteq (F_i'\sqcup A)\cup (F_j'\sqcup A).
    \]
    Obviously, $A \seq F_k$. So, we can find $F_k'\coloneqq F_k\setminus A\in \calF(\link_{\Delta}(A))$ such that $F_k'\succ' F_k$ with
    \[
    |F_j'\setminus F_k'| =1 \quad \text{ and } \quad F_i'\cap F_j' \subseteq F_k' \subseteq F_i' \cup F_j'.
    \]
    This shows that $\succ'$ is a strong shelling order of $\link_{\Delta}(A)$.
\end{proof}

\begin{example}
    Note that even when both $\link_{\Delta}(x)$ and $\Delta \setminus x$ are pure strongly shellable, $\Delta$ may not necessarily be strongly shellable. For example, let $\Delta$ be the complex whose facet set is
    \[
    \Set{\{1, 2, 3\}, \{2, 3, 4\}, \{3, 4, 5\}, \{4, 5, 6\}, \{5, 6, 7\}}.
    \]
    One can check directly that $\link_{\Delta}(7)$ and $\Delta \setminus 7$ are pure strongly shellable. On the other hand, notice that
    \[
    \dis_{\Delta}(\{1,2,3\},\{5,6,7\})=3,
    \]
    while
    \[
    \dis_{\Gamma(\Delta)}(\{1,2,3\},\{5,6,7\})=4.
    \]
    Thus, again by Theorem \ref{distance of graph less than complex}, we see that $\Delta$ is not strongly shellable.
\end{example}

\begin{proposition}
    \label{nonpure restriction}
    Assume that $\Delta$ is a strongly shellable  complex, and let $S$ be some subset of the set of vertices of $\Delta$. Assume furthermore that the induced complex $\Delta[S]$ satisfies the following condition:

    \begin{center}
        \begin{minipage}[h]{0.8\linewidth}
            if $\sigma$ is a maximal simplex in $\Delta$, then $\sigma \cap \Delta[S] = \sigma[S]$ is a maximal simplex in $\Delta[S]$.
        \end{minipage}
    \end{center}

    \noindent In this case, the complex $\Delta[S]$ is strongly shellable as well, and any strong shelling  order on $\calF(\Delta)$ induces a strong shelling order on $\calF(\Delta[S])$.
\end{proposition}

\begin{proof}
    Let $\succ$ be a strong shelling order on $\calF(\Delta)$. For each $F\in \calF(\Delta[S])$, let $\widetilde{F}$ be the first facet in $\calF(\Delta)$ with respect to $\succ$ such that $\widetilde{F}\cap S=F$. Then, we have the following induced order $\succ_S$ on $\calF(\Delta[S])$:
    \[
    \text{for each $F,G\in \calF(\Delta[S])$, $F\succ_S G$ if and only if $\widetilde{F}\succ \widetilde{G}$.}
    \]
    To check that $\succ_S$ is a strong shelling order, we take arbitrary $F,G\in \calF(\Delta[S])$ with $F\succ_S G$. Thus, $\widetilde{F}\succ \widetilde{G}$. As $\succ$ is a strong shelling order, one can find an $L' \in \calF(\Delta)$ such that $L' \succ \widetilde{G}$ and satisfies:
    \[
    |\widetilde{G} \setminus L'|=1 \quad \text{ and }\quad
    \widetilde{F}\cap \widetilde{G} \subseteq L' \subseteq \widetilde{F}\cup \widetilde{G}.
    \]
    Assume that $L \coloneqq L'\cap S$ and $\widetilde{G} \setminus L' = \{a\}$.

    We claim that $a \in G \setminus L$. In fact, it is easy to see that $a \notin L$ since $a \notin L'$. On the other hand, if $a \notin G$, then $G \seq L'\cap S$, i.e., $G \seq L$. Note that $L$ and $G$ are facets of $\Delta[S]$, so $L=G$. However, we will have $L'\cap S=\widetilde{G}\cap S$ and $L' \succ \widetilde{G}$. This contradicts to the assumption that $\widetilde{G}$ is the first facet in $\calF(\Delta)$ whose restriction to $S$ is $G$.

    Consequently, $a\in G\setminus L \subseteq \widetilde{G}\setminus L'=\Set{a}$, and hence $|G\setminus L|=1$. On the other hand, it is clear that $F\cap G\subseteq L \subseteq F\cup G$. Finally, $\widetilde{L}\succcurlyeq L'\succ \widetilde{G}$, hence $L\succ_S G$.
    Thus $\succ_S$ is a strong shelling order of $\Delta[S]$.
\end{proof}

Recall that if $\Gamma$  and $\Delta$ are two complexes over disjoint vertex sets, the \Index{join} of them is the complex on the vertex set $\calV(\Gamma)\sqcup \calV(\Delta)$:
\[
\Gamma * \Delta \coloneqq \Set{F\sqcup G | F\in \Gamma \text{ and }G\in \Delta}.
\]
It is well-known that  the join of two complexes is shellable if and only if each of the complexes is shellable; see \cite[Remark 10.22]{MR1401765}. We also have a strongly shellable version here.

\begin{proposition}
    \label{nonpure join}
    Let $\Gamma$ and $\Delta$ be two complexes. Then the join complex $\Gamma*\Delta$ is strongly shellable if and only if both $\Gamma$ and $\Delta$ are strongly shellable.
\end{proposition}

\begin{proof}
    Assume that $\Gamma$ and $\Delta$ are strongly shellable.
    Let $\succ_1$ and $\succ_2$ be two strong shelling orders on $\calF(\Gamma)$  and $\calF(\Delta)$,  respectively. Let $\succ$ be the lexicographic order with respect to $(\succ_1,\succ_2)$ on $\calF(\Gamma*\Delta)$, namely, for $F_1,F_2\in \calF(\Gamma)$  and $G_1,G_2\in\calF(\Delta)$, $F_1\sqcup G_1\succ F_2\sqcup G_2$ precisely whenever either $F_1\succ_1 F_2$, or $F_1=F_2$ and $G_1\succ_2 G_2$.

    To check the strong shellability of $\succ$, it suffices to take distinct $F_1,F_2\in \calF(\Gamma)$ and distinct $G_1,G_2\in \calF(\Delta)$ such that $F_1\sqcup G_1\succ F_2\sqcup G_2$. Then $F_1\succ_1 F_2$.  We can choose $F_3$ such that $F_3 \succ_1 F_2$ in $\calF(\Gamma)$, and satisfies:
    \[
    |F_2 \setminus F_3|=1 \quad \text{ and } \quad
    F_1 \cap F_2 \subseteq F_3 \subseteq F_1\cup F_2.
    \]
    It is clear that $F_3 \sqcup G_2 \succ F_2 \sqcup G_2$, and satisfies
    \[
    |(F_2 \sqcup G_2) \setminus (F_3 \sqcup G_2)|=1
    \]
    with
    \[
    (F_1\sqcup G_1) \cap (F_2\sqcup G_2) \subseteq (F_3\sqcup G_2) \subseteq (F_1\sqcup G_1) \cup (F_2\sqcup G_2).
    \]
    Thus $\succ$ is a strong shelling order of $\Gamma*\Delta$.

    For the other direction, we can apply Proposition \ref{nonpure restriction} by using $S=\calV(\Gamma)$ and $S=\calV(\Delta)$, respectively.
\end{proof}

Recently, Moradi and Khosh-Ahang \cite{arXiv:1601.00456} considered the expansion of simplicial complexes.

\begin{definition}
    Let $\Delta$ be a simplicial complex with the vertex set $\calV(\Delta)=\Set{x_1,\dots,x_n}$ and $s_1,\dots,s_n$ be arbitrary positive integers.  The \Index{$(s_1,\dots,s_n)$-expansion} of $\Delta$, denoted by $\Delta^{(s_1,\dots,s_n)}$, is the simplicial complex with the vertex set $\Set{x_{i,j}\mid 1\le i\le n, 1\le j\le s_i}$ and the facet set
    \[
    \Set{ \{x_{i_1,r_1},\dots,x_{i_{t},r_{t}}\} \mid \{x_{i_1},\dots,x_{i_t}\}\in \mathcal{F}(\Delta) \text{ and } (r_1,\dots,r_{t})\in [s_{i_1}]\times \cdots \times [s_{i_{t}}]}.
    \]
\end{definition}

Now, we wrap up this section with the following strongly shellable version of \cite[Corollary 2.15]{arXiv:1511.04676}.

\begin{theorem}
    \label{expansion-complex}
    Assume that $s_1, \dots, s_n$ are positive integers. Then $\Delta$ is strongly shellable if and only if  $\Delta^{(s_1,\dots,s_n)}$ is so.
\end{theorem}

\begin{proof}
    For simplicity, we will write $\Delta'$ for $\Delta^{(s_1,\dots,s_n)}$.

    \begin{enumerate}[a]
        \item The ``only if'' part:  By induction, it suffices to consider the special case when $s_k=1$ for $k=2,\dots,n$. Let $f: \calV(\Delta') \to \calV(\Delta)$ be the map by assigning each $x_{i,j}$ to $x_i$ and extend it to be a map from $\Delta'$ to $\Delta$.
            Let $\succ_\Delta$ be a strong shelling order on $\calF(\Delta)$. Consider two distinct facets $F_1,F_2\in \calF(\Delta')$.
            \begin{enumerate}[i]
                \item If $f(F_1)\ne f(F_2)$, we define $F_1\succ_{\Delta'} F_2$ if and only if $f(F_1)\succ_{\Delta} f(F_2)$.
                \item If $f(F_1)=f(F_2)$, then $F_1=\{x_{1,i},x_{i_2,1},\dots,x_{i_t,1}\}$ and $F_2=\{x_{1,j},x_{i_2,1},\dots,x_{i_t,1}\}$ for suitable $x_{i_2},\dots,x_{i_t} \in \calV(\Delta)\setminus \{x_1\}$. In this case, we define $F_1\succ_{\Delta'} F_2$ if and only if $i<j$.
            \end{enumerate}

            Now, we show that $\succ_{\Delta'}$ defines a strong shelling order on $\calF(\Delta')$. Indeed, we only need to check two distinct edges $F_1\succ_{\Delta'} F_2$ such that $f(F_1)\ne f(F_2)$. Whence, $f(F_1)\succ_{\Delta} f(F_2)$. Since $\Delta$ is strongly shellable, one can find $\widetilde{F}_3\in \calF(\Delta)$ such that $\widetilde{F}_3 \succ_{\Delta} f(F_2)$, $|f(F_2)\setminus \widetilde{F}_3|=1$ and $f(F_1)\cap f(F_2)\subseteq \widetilde{F}_3 \subseteq f(F_1)\cup f(F_2)$.

            We will construct a facet $F_3\in \calF(\Delta')$ with $f(F_3)=\widetilde{F}_3$ accordingly. Take arbitrary $x_{i}\in \widetilde{F}_3$ and we will find the $x_{i,j}$ for $F_3$.  Here, if $i\ne 1$, $j$ is forced to be $1$. If $x_1\in \widetilde{F}_3$, then $x_1\in f(F_1)\cup f(F_2)$. We have two sub-cases here. If $x_1\in f(F_2)$, we will take the $j$ such that $x_{1,j}\in F_2$. Otherwise, we will take the $j$ such that $x_{1,j}\in F_1$.

            In the following, we verify that $F_3$ is the expected facet proceeding $F_2$ with respect to $\succ_{\Delta'}$.
            \begin{enumerate}[1]
                \item As $f(F_3)=\widetilde{F}_3\succ_{\Delta} f(F_2)$, $F_3\succ_{\Delta'}F_2$.

                \item Now, suppose that $\{x_k\}=f(F_2)\setminus \widetilde{F}_3$. If $k\ne 1$, obviously $F_2\setminus F_3=\{x_{k,1}\}$. If $k=1$, then $F_2\setminus F_3=\{x_{1,j}\}$ for the unique $x_{1,j}\in F_2$. In both cases, $|F_2\setminus F_3|=1$.

                \item For $i\ne 1$, then $x_{i,1}\in F_3$ if and only if $x_i\in \widetilde{F}_3$. On the other hand, $x_{i,1}\in f(F_1)$ or $f(F_2)$ if and only if $x_i\in F_1$ or $F_2$ respectively. Thus
                    \[
                    F_1\cap F_2 \cap \{x_{i,1}\} \subseteq F_3\cap \{x_{i,1}\} \subseteq (F_1 \cup F_2)\cap \{x_{i,1}\}.
                    \]
                    For $i=1$, we may assume that $x_1\in f(F_1)\cup f(F_2)$.
                    Write $X_1=\Set{x_{1,1},\dots,x_{1,s_1}}$.
                    Depending on whether $x_1\in f(F_2)$ or not, we have $F_3\cap X_1=F_2\cap X_1$ or $F_1\cap X_1$. Thus,
                    \[
                    F_1\cap F_2 \cap X_1 \subseteq F_3\cap X_1 \subseteq (F_1 \cup F_2)\cap X_1.
                    \]
                    To sum up, we have $F_1\cap F_2 \subseteq F_3 \subseteq F_1\cup F_2$.
            \end{enumerate}
            Therefore, $\succ_{\Delta'}$ is a strong shelling order on $\Delta'$.

        \item The ``if'' part: Let $\succ_{\Delta'}$ be a strong shelling order on $\calF(\Delta')$. We will identify $x_{i,1}\in \calV(\Delta')$ with $x_i\in\calV(\Delta)$, treating $\Delta$ as a subcomplex of $\Delta'$. Thus, one has a natural induced linear order $\succ_{\Delta}$ on $\calF(\Delta)$. It suffices to show that $\succ_{\Delta}$ is a strong shelling order.

            Take two distinct facets $F_1\succ_{\Delta} F_2$ in $\calF(\Delta)\subset \calF(\Delta')$. Thus, by the strong shellability of $\Delta'$, one has $F_3\in \calF(\Delta')$ with $F_3\succ_{\Delta'}F_2$, $|F_2\setminus F_3|=1$ and $F_1\cap F_2\subseteq F_3 \subseteq F_1\cup F_2$. As $F_1\cup F_2 \subset \calV(\Delta)$, $F_3\subset \calV(\Delta)$. In particular, $F_3\in \calF(\Delta)$. Therefore, $F_3 \succ_\Delta F_2$. \qedhere
    \end{enumerate}
\end{proof}

\section{Codimension one graph}

Zheng \cite{MR2100472} considered the following property of simplicial complexes.

\begin{definition}
    A complex is called  \Index{connected in codimension one} if for any two facets $F$ and $G$ with $\dim(F) \geq \dim(G)$, there exists a chain of facets $F = F_0, \cdots, F_n = G$ between $F$ and $G$ such that
$\dim(F_i \cap F_{i+1})=\dim(F_{i+1})-1$ for all $i = 0,\cdots, n-1$.
\end{definition}

\begin{lemma}
    [{\cite[Lemma 9.1.12]{MR2724673}}]
    \label{lem-connected-1}
    Every Cohen-Macaulay complex is connected in codimension one.
\end{lemma}

As pure shellable complexes are Cohen-Macaulay by \cite[Theorem 8.2.6]{MR2724673}, the following simple result follows easily from the definition.

\begin{corollary}
    \label{connected in codimension one}
    Let $\Delta$ be a pure shellable complex. Then there exists a linear order on $\mathcal{F}(\Delta)$, say $F_1, \ldots, F_m$, such that $\braket{ F_1, \ldots, F_k }$ is connected in codimension one for each $k \in [m]$.
\end{corollary}

\begin{definition}
    \label{distance}
    For a complex $\Delta$, the \Index{distance} between two facets $F_1$ and $F_2$ is defined by
    \[
    \dis(F_1, F_2)\coloneqq \min(\dim(F_1), \dim(F_2))-\dim(F_1 \cap F_2),
    \]
    which can be easily verified to be $\min(|F_1 \setminus F_2|, |F_2 \setminus F_1|)$.
    Sometimes, we will also write it as $\dis_{\Delta}(F_1, F_2)$, in order to emphasize the underlying complex.
\end{definition}

Note that in the pure case, the function $\dis_\Delta$  satisfies the usual triangle inequality. However, this is generally false in the nonpure case.

\begin{lemma}
    \label{nonpure subset}
    Let $\Delta$ be a complex. If $F,G,H$ are three facets such that
    \[
    \min(\dim(F), \dim(G)) \leq \dim(H) \leq \max(\dim(F), \dim(G)),
    \]
    then
    $\dis(F,H)+\dis(H,G)=\dis(F,G)$ if and only if $F\cap G \subseteq H \subseteq F\cup G$.
\end{lemma}

\begin{proof}
    As indicated in Figure \ref{fig: nonpure same}, we may assume that $F=P_1\sqcup P_2 \sqcup P_4 \sqcup P_5$, $G=P_1\sqcup P_3 \sqcup P_4 \sqcup P_7$ and $H=P_1\sqcup P_2 \sqcup P_3 \sqcup P_6$.
    \begin{figure}[!ht]
        \begin{minipage}[h]{\linewidth} \centering
            \includegraphics{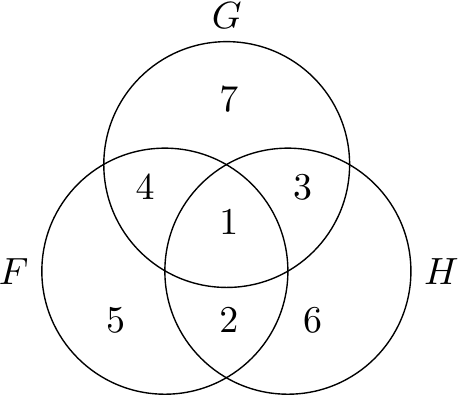}
        \end{minipage}
        \caption{Venn diagram of 3 intersecting sets} \label{fig: nonpure same}
    \end{figure}
    Suppose that $\dim(F)\ge \dim(G)$. Therefore, indeed, $\dim(F)\ge \dim(H)\ge \dim(G)$. In this case, $\dis(F,H)=|H\setminus G|=|P_3|+|P_6|$, $\dis(H,G)=|G\setminus H|=|P_4|+|P_7|$ and $\dis(F,G)=|G\setminus F|=|P_3|+|P_7|$. Thus, the condition $\dis(F,H)+\dis(H,G)=\dis(F,G)$ is translated into $|P_4|+|P_6|=0$, i.e., $P_4=P_6=\varnothing$. But this is equivalent to saying that $F\cap G \subseteq H \subseteq F\cup G$.
\end{proof}

\begin{corollary}
    \label{nonpure same}
    Under the assumptions in Proposition \ref{nonpure restriction}, let $F_1,F_2$ be two distinct facets of $\Delta$ such that $F_1\cap S=F_2\cap S$. If $G$ is another facet such that
    \[
    \min(\dim(F_1), \dim(F_2)) \leq \dim(G) \leq \max(\dim(F_1), \dim(F_2))
    \]
    and
    \[
    \dis(F_1,G)+\dis(G,F_2)=\dis(F_1,F_2),
    \]
    then $G\cap S=F_1\cap S$.
\end{corollary}

\begin{proof}
    Without loss of generality, we assume that $\dim(F_1) \geq \dim(G) \geq \dim(F_2)$.
    By Lemma \ref{nonpure subset}, we have $F_1 \cap F_2 \seq G \seq F_1 \cup F_2$, and thus $(F_1\cap S) \cap (F_2\cap S) \seq G\cap S \seq (F_1\cap S) \cup (F_2\cap S)$. As $F_1\cap S=F_2\cap S$, we have $G\cap S=F_1\cap S$.
\end{proof}

In the above corollary, the condition on dimensions
is necessary, as the following example shows.

\begin{example}
    Let $\Delta$ be a complex on the set $\{1, 2, 3, 4, 5, 6, 7\}$ whose facet set is
    \[
    \Set{F_1=\{1, 2, 3, 4\}, \quad F_2=\{4, 5, 6\},\quad G=\{1, 2, 3, 5, 6, 7\}}.
    \]
    It is direct to check that $\dis(F_1,G)+\dis(G,F_2)=\dis(F_1,F_2)$. If $S= \{4, 7\}$, then $F_1\cap S=F_2\cap S= \{4\}$, but $G\cap S=\{7\} \neq F_1\cap S$.
\end{example}

Note that every finite graph $\Gamma$ has a well-defined distance function defined from $\calV(\Gamma)^2$ to $\NN\cup \{+\infty\}$. We will denote it by $\dis_\Gamma$.

\begin{definition}
    Given a complex $\Delta$, the \Index{codimension one graph} related to $\Delta$, denoted by $\Gamma(\Delta)$, is a finite simple graph whose vertex set is $\mathcal{F}(\Delta)$, and two facets $F$ and $G$ are adjacent in $\Gamma(\Delta)$ if and only if $\dis_{\Delta}(F, G)=1$.
    The codimension one graph $\Gamma(\Delta)$ will be called \Index{harmonious} with respect to $\Delta$ if
    $\dis_\Delta(F,G)=\dis_{\Gamma(\Delta)}(F,G)$ for every pair of facets $F$ and $G$.  In this case, we also call $\Delta$ a \Index{harmonious} complex.
\end{definition}

\begin{example}
    \label{example:3.2}
    Let $\Delta$ be the complex whose facets are $\{1, 2\}$, $\{2, 3\}$, $\{3, 4\}$ and $\{4, 5\}$. The codimension one graph of $\Delta$ is pictured in Figure \ref{Fig:codim1}.  For simplicity, in this figure, we write $x_{i,j}$ for the vertex corresponding to the facet $\{i,j\}$.
    \begin{figure}[!ht]
        \begin{minipage}[h]{\linewidth} \centering
            \begin{minipage}[h]{\linewidth} \centering
                \includegraphics{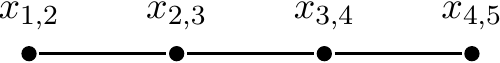}
            \end{minipage}       
            \caption{Codimension one graph of $\Delta$} \label{Fig:codim1}
        \end{minipage}
    \end{figure}

    Note that
    \[
    \dis_{\Gamma(\Delta)}(\{1, 2\},\, \{4, 5\})=3\ne \dis_{\Delta}(\{1, 2\},\, \{4, 5\})=2.
    \]
    Hence $\Delta$ is not a harmonious complex.
\end{example}

\begin{lemma}
    \label{lem-connected-2}
    If $\Delta$ is a Cohen-Macaulay complex, then $\Gamma(\Delta)$ is connected.
\end{lemma}

\begin{proof}
    This is simply a paraphrase of Lemma \ref{lem-connected-1}.
\end{proof}

We will have more control regarding the connectivity when strong shellability is present.

\begin{lemma}
    \label{distance of graph less than complex}
    Let $\Delta$ be a strongly shellable complex. Then $\dis_{\Gamma(\Delta)}(F, G) \leq \dis_{\Delta}(F, G)$  for each pair of facets $F, G \in \Delta$. In particular, $\Gamma(\Delta)$ is connected.
\end{lemma}

\begin{proof}
    Let $F_1, F_2, \cdots, F_t$ be a dimension-decreasing strong shelling order. We will prove by induction on the distance $\dis_{\Delta}(F, G)$.
    The case $\dis_\Delta(F,G)=0$ is trivial. Thus, we will assume that $F\ne G$ and $\dis_\Delta(F,G)>0$.
    Let $F=F_i$ and $G=F_j$ with $i<j$. Then there exists $k<j$, such that $|F_j \setminus F_k|=1$ and $F_i\cap F_j \subseteq F_k \subseteq F_i\cup F_j$. Obviously, $\dis_{\Delta}(F_k, F_j)=1=\dis_{\Gamma(\Delta)}(F_k, F_j)$.  We have the following two subcases:
    \begin{enumerate}[1]
        \item $\dim(F_k) \leq \dim(F_i)$. In the case,  $\dis_{\Delta}(F_i, F_k)=|F_k \setminus F_i|$ and $\dis_{\Delta}(F_i, F_j)=|F_j \setminus F_i|$. Since $|F_j \setminus F_i|= |F_k \setminus F_i|+1$ by Remark \ref{1/2/3 equivalent discription}, it follows that $\dis_{\Delta}(F_i, F_k)=\dis_{\Delta}(F_i, F_j)-1$.
        \item \label{lemma-case-2} $\dim(F_k) > \dim(F_i)$. In the case,  $\dis_{\Delta}(F_i, F_k)=|F_i \setminus F_k| < |F_k \setminus F_i|$. Since $\dis_{\Delta}(F_i, F_j)=|F_j \setminus F_i|= |F_k \setminus F_i|+1$, it follows that $\dis_{\Delta}(F_i, F_k) < \dis_{\Delta}(F_i, F_j)-1$.
    \end{enumerate}

    By inductive assumption, $\dis_{\Gamma(\Delta)}(F_i, F_k) \leq \dis_{\Delta}(F_i, F_k)$. Hence
    \[
    \dis_{\Gamma(\Delta)}(F_i, F_j) \leq \dis_{\Gamma(\Delta)}(F_i, F_k)+1 \leq \dis_{\Delta}(F_i, F_k)+1 \leq \dis_{\Delta}(F_i, F_j). \qedhere
    \]
\end{proof}

The above proof indicates that the appearance of $F_k$ in the case \ref{lemma-case-2} induces the strict inequality $\dis_{\Gamma(\Delta)}(F_i, F_j) < \dis_{\Delta}(F_i, F_j)$. Thus, we have the following result.

\begin{lemma}
    \label{not add dim}
    Let $\Delta$ be a strongly shellable complex, and $\succ$ be a dimension-decreasing strong shelling order on $\mathcal{F}(\Delta)$. If $\Gamma(\Delta)$ is harmonious with respect to $\Delta$, then for each pair $F_i\succ F_j$, there exists a facet $F_k\succ F_j$, such that:
    \begin{enumerate}[1]
        \item \label{nad-1} $|F_j \setminus F_k|=1$ and $F_i\cap F_j \subseteq F_k \subseteq F_i\cup F_j$;
        \item \label{nad-2} $\dim(F_i)\ge \dim(F_k)\ge \dim(F_j)$.
    \end{enumerate}
    In particular, $\dis(F_i,F_j)=\dis(F_i,F_k)+\dis(F_k,F_j)$ with $\dis(F_k,F_j)=1$.
\end{lemma}

\begin{remark}
    \label{strong-harmonious}
    As a matter of fact, if $\Delta$ is harmonious, then it is clear from the proof of Lemma \ref{distance of graph less than complex} that any $F_k$ satisfying the condition \ref{nad-1} in Lemma \ref{not add dim} will automatically satisfy the condition \ref{nad-2} as well.
\end{remark}

\begin{definition}
A complex $\Delta$ will be called \Index{quasi-harmonious}, if there exists a dimension-decreasing strong shelling order $\succ$ on $\calF(\Delta)$, such that for each pair $F_i\succ F_j$, there exists a facet $F_k\succ F_j$ such that $F_k$ satisfies the conditions \ref{nad-1} and \ref{nad-2} in Lemma \ref{not add dim}.
\end{definition}

Obviously, strongly shellable harmonious complexes are quasi-harmonious.
Later, in Theorem \ref{strongly shellable=keeping distance order}, we will show that pure strongly shellable complexes are harmonious. Hence pure strongly shellable complexes are always quasi-harmonious. But in the nonpure case, a strongly shellable complex may not necessarily be quasi-harmonious.

\begin{example}
    [Strongly shellable, but not quasi-harmonious]
    Let $\Delta$ be a complex with facets
    \[
    \{3, 4, 5, 6, 7\}, \{2, 4, 5, 6, 7\}, \{2, 3, 5, 6, 7\}, \{2, 3, 4, 6, 7\}, \{2, 3, 4, 5\}, \{1, 6, 7\}.
    \]
    One can check directly that $\Delta$ is a nonpure strongly shellable complex. But for the facets $\{2, 3, 4, 5\}$ and $\{1, 6, 7\}$, the condition for being quasi-harmonious is not satisfied.
\end{example}

On the other hand, nonpure quasi-harmonious complexes are generally not harmonious.

\begin{example}
    [Quasi-harmonious, but not harmonious]
   Let $\Delta$ be a complex with facets
   \[
   \{2,3,5,6\}, \{1,2,3\}, \{2,3,4\},\{3,4,5\},\{4,5,6\}.
   \]
   One can check directly that this strongly shellable complex is a quasi-harmonious, but not harmonious.
\end{example}

However, the distance function does behave more tamely for quasi-harmonious complexes.

\begin{proposition}
    \label{cor:dist-1}
    Suppose that the simplicial complex $\Delta$ is quasi-harmonious with respect to the dimension-decreasing strong shelling order $\succ$ on $\mathcal{F}(\Delta)$.  Then, for each pair of facets $F_i\succ F_j$, there exists a sequence of facets $F_i=G_0,G_1,\dots,G_t=F_j$ such that
    \begin{enumerate}[a]
        \item $t=\dis(F_i,F_j)$;
        \item $\dim(G_0)\ge \dim(G_1)\ge \cdots \ge \dim(G_t)$;
        \item $\dis(G_h,G_{h+1})=1$ for each $h$ with $0\le h \le t-1$.
    \end{enumerate}
\end{proposition}

\begin{proof}
    Assume that $t=\dis(F_i,F_j)$ and $F_k$ is the facet as in Lemma \ref{not add dim}.  It suffices to consider the case when $t\ge 2$. Whence, $F_i\ne F_k$.  Denote $F_i,F_k$ and $F_j$ by $G_0,G_{t-1}$ and $G_t$ respectively.
    \begin{enumerate}[1]
        \item If $F_i\succ F_k$, by induction, we will have a sequence of facets $G_0,G_1,\dots,G_{t-1}$ such that $\dim(G_0)\ge \dim(G_1)\ge \cdots \ge \dim(G_{t-1})$, and $\dis(G_h,G_{h+1})=1$ for each $h$ with $0\le h \le t-2$.
        \item If $F_k\succ F_i$, then $\dim(F_i)=\dim(F_k)$. Again by induction, we will have a sequence of facets $G_0,G_1,\dots,G_{t-1}$ of same dimension with $\dis(G_h,G_{h+1})=1$ for each $h$ with $0\le h \le t-2$.
    \end{enumerate}

    In either case, we are done after concatenating the sequence with $G_t$ in the end.
\end{proof}

The harmonious property behaves well with respect to restrictions of strong shellability.

\begin{lemma}
    \label{harmonious subcomplex}
    Let $\Delta$ be a harmonious strongly shellable complex. Let $S$ be a non-empty subset of $\calF(\Delta)$ and $\Delta'=\braket{S}$. If $\Delta'$ is also strongly shellable, then it is harmonious as well.
\end{lemma}

\begin{proof}
    Take arbitrary $F_a,F_b\in \calF(\Delta')$. As any minimal path connecting $F_a$ and $F_b$ in $\Gamma(\Delta')$ is a path connecting $F_a$ and $F_b$ in $\Gamma(\Delta)$, one has $\dis_{\Gamma(\Delta')}(F_a,F_b)\ge \dis_{\Gamma(\Delta)}(F_a,F_b)=\dis(F_a,F_b)$. But $\dis_{\Gamma(\Delta')}(F_a,F_b)\le \dis(F_a,F_b)$ by Lemma \ref{distance of graph less than complex}. Consequently, $\dis_{\Gamma(\Delta')}(F_a,F_b)= \dis(F_a,F_b)$.
\end{proof}

For a given linear order $\succ$ on $\mathcal{F}(\Delta)$, a facet $F$ of dimension $k$ is called an \Index{initial facet}, if $F \succ G$ for any other facet $G$ of same dimension.
One can similarly define \Index{terminal facet}.
Meanwhile, for any facet $F_1,F_2$, we have the \Index{interval} $[F_1, F_2]_\succ \coloneqq \Set{G \in \mathcal{F}(\Delta) \mid F_1 \succcurlyeq G \succcurlyeq F_2}$.

\begin{proposition}
    \label{interval ss}
    Let $\Delta$ be a harmonious  strongly shellable  complex with a dimension-decreasing strong shelling order $\succ$ on $\mathcal{F}(\Delta)$. If $F_i$ is an initial facet of $\succ$, then for every $F_j$ with $F_i\succ F_j$, the interval $[F_i, F_j]_\succ$ generates a harmonious strongly shellable complex.
\end{proposition}

\begin{proof}
    By Lemma \ref{not add dim}, for each pair of facets $F_a,  F_b \in [F_i, F_j]_\succ$ with $F_a \succ F_b$, there exists a facet $F_c \succ F_b\succcurlyeq F_j$, such that $|F_b \setminus F_c|=1$, $F_a\cap F_b\subseteq F_c \subseteq F_a\cup F_b$, and $\dim(F_c) \leq \dim(F_a)$. Since $F_i$ is an initial facet,
    and $\succ$ is dimension-decreasing, $\dim F_a\le \dim F_i$. Thus, $F_c \in [F_i, F_j]_\succ$. This shows that $[F_i, F_j]_\succ$ generates a strongly shellable complex. As for the expected harmonious property, we apply Lemma \ref{harmonious subcomplex}.
\end{proof}

\begin{corollary}
    \label{pure_k F hss}
    Let $\Delta$ be a harmonious strongly shellable  complex. Then $\overline{\pure_k}(\Delta)$ is also harmonious strongly shellable for any $0 < k \leq \dim(\Delta)$.
\end{corollary}

\begin{proof}
    Take arbitrary dimension-decreasing strong shelling order $\succ$ on $\calF(\Delta)$.  Let $F_i$ and $F_j$ be the initial and terminal facet of dimension $k$ respectively. Then $\calF(\overline{\pure_k}(\Delta))=[F_i,F_j]_\succ$. Now, we apply Proposition \ref{interval ss}.
\end{proof}

\section{Strongly shellable complexes, pure case}

In this section, we focus on pure strongly shellable complexes.
In this case, the distance function $\dis_\Delta$ in Definition \ref{distance} satisfies the usual triangle inequality:
\[
\dis(F,H)+\dis(H,G)\ge\dis(F,G) \quad \text{for $F,H,G\in \calF(\Delta)$}.
\]

The following lemma is the pure version of Lemma \ref{nonpure subset}.

\begin{lemma}
    \label{subset}
    Let $F,G,H \in \binom{[n]}{d}$. Then $\dis(F,H)+\dis(H,G)=\dis(F,G)$ if and only if $F\cap G \subseteq H \subseteq F\cup G$.
\end{lemma}

The following fact follows directly from the definition and the above lemma:

\begin{lemma}
    \label{d=d-1+1}
    Let $\Delta$ be a pure complex. Then $\Delta$ is  strongly shellable if and only if there exists a linear order $\succ$ on $\mathcal{F}(\Delta)$, such that whenever $F_i\succ F_j$, there exists a facet $F_k\succ F_j$, such that $\dis(F_k, F_j)=1$ and $\dis(F_i, F_k) = \dis(F_i, F_j)-1$.
\end{lemma}

Given a linear order $\succ$ on $\calF(\Delta)$, if $F\succcurlyeq G$, we will say that $F$ is \Index{on the left side} of $G$. Similar to the phenomenon in Proposition \ref{cor:dist-1}, pure strongly shellable complexes can be further characterized by the distance function as follows.

\begin{proposition}
    \label{distance 1 chain}
    The following statements are equivalent for a pure simplicial complex $\Delta$:
    \begin{enumerate}[a]
        \item \label{dist-1-a} $\Delta$ is strongly shellable.
        \item \label{dist-1-b} There exists a linear order $\succ$ on $\mathcal{F}(\Delta)$, such that for each pair $F,G \in \mathcal{F}(\Delta)$  with $F\succ G$, there exists a chain of facets $F= H_{0}, \ldots, H_{t} = G$ of length $t=\dis(F,G)$ on the left side of $G$ with $\dis(F_{l-1}, F_{l})=1$ for $1 \leq l \leq t$.
        \item \label{dist-1-c} There exists a linear order $\succ$ on $\mathcal{F}(\Delta)$, such that for each pair $F,G \in \mathcal{F}(\Delta)$  with $F\succ G$, there exists a chain of facets $F = H_{0}, \ldots, H_{t} = G$ of length $t=\dis(F,G)$ on the left side of $G$ with $\dis(H_{l_1}, H_{l_2})=l_2-l_1$ for $0\le l_1 \le l_2 \le t$.
        \item \label{dist-1-d} There exists a linear order $\succ$ on $\mathcal{F}(\Delta)$, such that for each pair $F, G \in \mathcal{F}(\Delta)$  with $F\succ G$, either $\dis(F,G)=1$, or $\dis(F,G)\ge 2$ and there exists some $H\succ G$ with $F\ne H$ such that $\dis(F,H)+\dis(H,G)=\dis(F,G)$.
        \item \label{dist-1-e} There exists a linear order $\succ$ on $\mathcal{F}(\Delta)$, such that for each pair $F,G\in \mathcal{F}(\Delta)$  with $F\succ G$, either $\dis(F,G)=1$, or $\dis(F,G)\ge 2$ and there exists some $H\succ G$ with $F\ne H$ such that $F \cap G \seq H \seq F \cup G$.
    \end{enumerate}
\end{proposition}

\begin{proof}
    The implications \ref{dist-1-d} $\Leftrightarrow$ \ref{dist-1-e}
    follows from Lemma \ref{subset}.
    The implications \ref{dist-1-d} $\Leftarrow$ \ref{dist-1-a} $\Leftarrow$ \ref{dist-1-b} $\Leftarrow$ \ref{dist-1-c} are clear from Lemma \ref{d=d-1+1}.

    \begin{enumerate}[i]
        \item The implication \ref{dist-1-a} $\Rightarrow$ \ref{dist-1-c} is straightforward as follows. Assume that $F\succ G$. If $t=\dis(F,G)=1$, this is clear. Thus, we may assume that $t\ge 2$.
            Write $F=H_0$ and $G=H_t$.
            Therefore, there exists $H_{t-1}\succ H_t$ with $\dis(H_{t-1},H_t)=1$ and $\dis(H_0,H_{t-1})=t-1$.

            If $H_0\succ H_{t-1}$, by induction, we can find $H_0,H_1,\dots,H_{t-1}$ on the left side of $H_{t-1}$ such that $\dis(H_{l_1},H_{l_2})=l_2-l_1$ for $0\le l_1\le l_2\le t-1$. Now take $0\le l \le t-1$ and check $\dis(H_l,H_t)$. Obviously, $\dis(H_l,H_t)\le \dis(H_l,H_{t-1})+\dis(H_{t-1},H_t)=(t-1-l)+1=t-l$. On the other hand, $\dis(H_l,H_t)\ge \dis(H_0,H_t)-\dis(H_0,H_l)=t-l$. Thus, $\dis(H_l,H_t)=t-l$, as expected.

            The case when $H_{t-1}\succ H_{0}$ is similar. By induction, we will find a chain of facets $H_{t-1}=H_0',H_1',\dots,H_{t-1}'=H_0$ on the left of $H_0$ with $\dis(H_{l_1}',H_{l_2}')=l_2-l_1$ for $0\le l_1\le l_2 \le t-1$. We will take $H_{l}=H_{t-1-l}'$ for $1\le l\le t-2$. Notice that $H_0=H_{t-1}'$ and $H_{t-1}=H_0'$ already. The remaining argument is similar to that of the above.

        \item As a final step, we show \ref{dist-1-d} $\Rightarrow$ \ref{dist-1-a}, which is also easy. Take arbitrary $F\succ G$.
            We show that there exists some $L$ with $L\succ G$, $\dis(L,G)=1$ and $\dis(F,L)=\dis(F,G)-1$. To show this,
            we may assume that $t=\dis(F,G)\ge 2$. Therefore, there exists some $H\succ G$ with $F\ne H$ and $\dis(F,H)+\dis(H,G)=t$. As $1\le \dis(H,G)\le t-1$, by induction, we can find $L\succ G$ such that $\dis(L,G)=1$ and $\dis(H,L)=\dis(H,G)-1$. Now, $\dis(F,L)\le \dis(F,H)+\dis(H,L)=(t-\dis(H,G))+(\dis(H,G)-1)=t-1$. On the other hand, $\dis(F,L)\ge \dis(F,G)-\dis(L,G)=t-1$. Therefore, $\dis(F,L)=t-1$, as expected. \qedhere
    \end{enumerate}
\end{proof}

In the following, we will characterize the strong shellability of a pure complex by investigating its codimension one graph.

\begin{definition}
    For a given finite simple graph $G$, we can delete a vertex $v_1$ to get a subgraph $G_1=G\setminus v_1$. If for each pair of vertices $u,\,v$ in $G_1$, $\dis_{G_1}(u, v)=\dis_{G}(u, v)$, the graph $G_1$ is called \Index{distance-preserving} with respect to $G$. More generally, if we can order the vertex set $\calV(G)=\Set{v_1,v_2,\dots,v_{t}}$, such that the induced subgraphs $G_k=G|_{\{v_{k+1},v_{k+2},\dots,v_{t}\}}$ satisfy the requirement that
    $G_i$  preserves distance with respect to $G_{i-1}$ for each $1 \leq i < t-1$, then we say that the graph $G$ has a \Index{distance-preserving order}: $v_1, v_2, \cdots, v_{t}$.
\end{definition}

\begin{example}
    Let $\Delta$ be a pure complex with the facet set
    \[
    \Set{\{1, 2, 3, 4\}, \{2, 3, 4, 5\}, \{3, 4, 5, 6\}, \{4, 5, 6, 7\}, \{1, 4, 6, 7\}, \{1, 2, 4, 7\}}.
    \]
    The codimension one graph $\Gamma(\Delta)$ is a cycle with 6 vertices, as Figure \ref{Fig:circle} shows.
    It is easy to check that $\Gamma(\Delta)$ is harmonious  
    but does not have any distance-preserving  order.   
\end{example}

\begin{figure}[!ht]
    \begin{minipage}[h]{\linewidth} \centering
        \begin{minipage}[h]{\linewidth} \centering
            \includegraphics{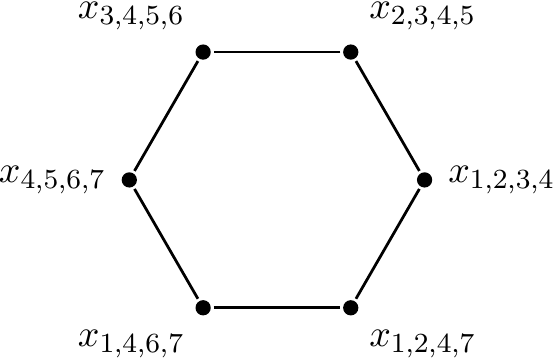}
        \end{minipage}
        \caption{$\Gamma(\Delta)$}
        \label{Fig:circle}
    \end{minipage}
\end{figure}

As $\dis_\Delta$ is indeed a distance metric on $\calF(\Delta)$ in pure case, we have the following lemma.

\begin{lemma}
    \label{pure distance of complex less than graph}
    Let $\Delta$ be a pure complex. Then $\dis_{\Delta}(F, G) \leq \dis_{\Gamma(\Delta)}(F, G)$ for each pair of facets $F, G \in \Delta$.
\end{lemma}

\begin{proof}
    Assume that $\dis_{\Gamma(\Delta)}(F, G)=t$. If $F=F_0,F_1,\dots,F_t =G$ is a path of length $t$ in $\Gamma(\Delta)$, then,  for each $1 \leq i \leq t$, $|F_{i-1} \cap F_i| =d-1$, i.e., $\dis_\Delta(F_{i-1},F_i)=1$. By the triangle inequality, $\dis_\Delta(F,G)\le t$.
\end{proof}

The following is the first main result of this section.

\begin{theorem}
    \label{strongly shellable=keeping distance order}
    A pure complex $\Delta$ is strongly shellable if and only if its codimension one graph $\Gamma(\Delta)$ is harmonious and has a distance-preserving order.
\end{theorem}

\begin{proof}
    For the necessity part: If $\Delta$ is strongly shellable, then for each pair of facets $F, G \in \Delta$,  $\dis_{\Gamma(\Delta)}(F, G) \leq \dis_{\Delta}(F, G)$ by Lemma \ref{distance of graph less than complex}. On the other hand, by Lemma \ref{pure distance of complex less than graph}, $\dis_{\Delta}(F, G) \leq \dis_{\Gamma(\Delta)}(F, G)$. Hence $\Gamma(\Delta)$ is harmonious.

    Since $\Delta$ is strongly shellable, we can assume that $G_1, G_2, \dots, G_s$ is a strong shelling order on $\calF(\Delta)$. Note that each subcomplex $\braket{ G_1, \dots, G_k}$ for $1 \leq k \leq s$ is also strongly shellable, therefore harmonious, by the previous argument. It is clear that $G_s, G_{s-1}, \cdots, G_1$ gives a distance-preserving order on $\Gamma(\Delta)$.

    For the sufficiency part: Assume that $\Gamma(\Delta)$ has a distance-preserving order: $G_s, G_{s-1}, \dots, G_1$. Since $\Gamma(\Delta)$ is harmonious, the subcomplexes $\Delta_k\coloneqq\braket{G_1, \cdots, G_k}$ are also harmonious for $1 \leq k \leq s$. In the following, we will show that $G_1, \cdots, G_s$ is a strong shelling order on $\calF(\Delta)$. Take arbitrary facets $G_i, G_j \in \Delta$ with $i<j$ and suppose that $\dis_{\Delta}(G_i, G_j)=t$.  Let $\Gamma_j=\Gamma(\Delta_j)$.  Since $\Delta_j$ is harmonious, $\dis_{\Gamma_j}(G_i,G_j)=\dis_{\Delta_j}(G_i,G_j)=t$.  Therefore, there exists a path of length $t$ in $\Gamma_j$:  $G_i=G_{k_0},G_{k_1}, \dots,G_{k_t}=G_j$, with all $k_l\le j$. Obviously, $\dis_{\Delta}(G_{k_{t-1}}, G_j) = \dis_{\Gamma_j}(G_{k_{t-1}}, G_j)=1$. Note that $ \dis_{\Delta}(G_{k_{t-1}}, G_i)\geq  \dis_\Delta(G_i,G_j)-\dis_{\Delta}(G_j,G_{k_{t-1}})=t-1$. On the other hand, $\dis_{\Gamma_j}(G_{k_{t-1}}, G_i) \leq t-1$ by the existence of the previous path connecting $G_i$ and $G_j$. Hence, by Lemma \ref{pure distance of complex less than graph}, we have $\dis_{\Delta}(G_{k_{t-1}}, G_i)=\dis_{\Gamma_j}(G_{k_{t-1}}, G_i) = t-1$. It follows from Lemma \ref{d=d-1+1} that $G_1, \dots, G_s$ is a strong shelling order of $\Delta$.
\end{proof}

Recall that a cycle in a finite simple graph $G$ is called \Index{minimal} if there is no chord in the cycle. And the \Index{girth} of $G$, denoted by $\g(G)$, is the length of a shortest cycle contained in the graph.
If a graph has no cycle, then its girth  is assumed to be 0.

\begin{lem}
    \label{girth not more than 4}
    Let $G$ be a connected simple graph. If it has a distance-preserving order, then its girth is at most $4$.
\end{lem}

\begin{proof}
    We may assume that $G$ has at least one cycle, and  has a distance-preserving order: $v_1, v_2, \dots, v_{n+1}$. Let $G_i$ be the induced subgraph from $G$ by removing the vertices $v_1,\dots,v_i$.

    Assume for contradiction that $\g(G)\ge 5$.
    As $G_{n-1}$ has no cycle, we can find the least $k$ such that $G_k$ has no cycle. Obviously $1\le k \le n-1$. From $G_{k-1}$ to $G_k$, we removed the vertex $v_k$. Thus, $v_k$ is contained in some minimal cycle in $G_{k-1}$. Consider one such minimal cycle. Let $v'$ and $v''$ be the two vertices adjacent to $v_k$ on this cycle. As $\g(G)\ge 5$, and $G_{k-1}$ contains cycles, $\g(G_{k-1})\ge 5$. Thus, $v'$ and $v''$ are not adjacent in $G_{k-1}$ and indeed $\dis_{G_{k-1}}(v',v'')=2$. By the distance-preserving condition, $\dis_{G_k}(v',v'')=2$. This implies the existence of some vertex $u\in G_k$ which are adjacent to both $v'$ and $v''$. But $u\ne v_{k}$. Thus we have a cycle in $G_{k-1}$ consisting of the vertices $v_k, v', u$ and $v''$. This implies that  $\g(G_{k-1})\le 4$, a contradiction.
\end{proof}

By Theorem \ref{strongly shellable=keeping distance order} and Lemma \ref{girth not more than 4}, we have the following result.

\begin{proposition}
    \label{girth diameter}
    If $\Delta$ is a pure strongly shellable complex,
    then its codimension one graph $\Gamma(\Delta)$ is connected, and $\g(\Gamma(\Delta)) \leq 4$, $\diam(\Gamma(\Delta)) \leq \dim(\Delta)+1$.
\end{proposition}

We will wrap up this section with a very important property of pure strongly shellable complexes.
We have already mentioned that any pure strongly shellable complex $\Delta$ (say, over the vertex set $[n]$) is Cohen-Macaulay, i.e., the Stanley-Reisner ring $\KK[\Delta]= \KK[x_1, \ldots, x_n]/I_{\Delta}$ is a Cohen-Macaulay ring over arbitrary field $\KK$.  This is an important property related to the Stanley-Reisner ideal $I_\Delta$ of $\Delta$. In the following, we will consider the \Index{facet} ideal $I(\Delta)\coloneqq\braket{\bdx^F\mid F\in \calF(\Delta)}\subseteq \KK[x_1,\dots,x_n]$.  Here, for any $F\subseteq [n]$, we write $\bdx^F\coloneqq\prod_{i\in F}x_i$.
Notice that for any
given simplicial complex $\Delta$, its \Index{complement complex} $\Delta^c$ has the facet set $\calF(\Delta^c)=\Set{F^c: F\in \calF(\Delta)}$, where $F^c\coloneqq\calV(\Delta)\setminus F$. The following observation is important.

\begin{lemma}
    \label{complement shellable}
    A pure complex $\Delta$ is strongly shellable if and only if its complement complex $\Delta^{c}$ has the same property.
\end{lemma}

\begin{proof}
    Note that  $\Delta^c$ is also a pure complex with  $(\Delta^{c})^{c}=\Delta$. Furthermore, for each pair of facets $A, B$ of $\Delta$, $\dis_{\Delta^{c}}(A^{c}, B^{c})=\dis_{\Delta}(A, B)$.  Thus, the result follows directly from Lemma \ref{d=d-1+1}.
\end{proof}

Let $S=\KK[x_1,\dots,x_n]$ be a polynomial ring over a field $\KK$ and $I$ a graded proper ideal.  Recall that $I$ has \Index{linear quotients}, if there exists a system of homogeneous generators $f_1 ,f_2 ,\dots,f_m$ of $I$ such that the colon ideal $\braket{f_1 ,\dots,f_{i-1}} : f_i$ is generated by linear forms for all $i$. If $I$ has linear quotients, then $I$ is componentwise linear; see \cite[Theorem 8.2.15]{MR2724673}. In particular, if $I$ has linear quotients and can be generated by forms of degree $d$, then it has a $d$-linear resolution; see \cite[Proposition 8.2.1]{MR2724673}.

On the other hand, recall that the \Index{Alexander dual} of $\Delta$ (with respect to the vertex set $\calV(\Delta)$), denoted by $\Delta^{\vee}$, is the complex $\Delta^{\vee}\coloneqq\Set{F^c \mid F \notin \Delta,\, F\subseteq [n]}$.  Note that $I_{\Delta^{\vee}} = I(\Delta^{c})$ by \cite[Lemma 1.5.3]{MR2724673}. Thus,  we have the second main result of this section.

\begin{theorem}
    \label{linear quotients}
    If $\Delta$ is a pure strongly shellable complex, then the facet ideal $I(\Delta)$ has linear quotients.
\end{theorem}

\begin{proof}
    It is well-known that $\Delta$ is shellable if and only if $I_{\Delta^{\vee}}$ has linear quotients; see \cite[Proposition 8.2.5]{MR2724673}. Since $\Delta$ is strongly shellable, $\Delta^{c}$ is shellable by Lemma \ref{complement shellable}. Hence $I_{(\Delta^{c})^{\vee}}$ has linear quotients. Note that  $I_{(\Delta^{c})^{\vee}} = I(\Delta)$. This completes the proof.
\end{proof}

This important property will be vital for the sequel paper \cite{ESSC}, where we deal with chordal (hyper)-graphs.

\section{Strongly shellable posets}
Recall that a poset $P$ is called \Index{bounded} if it has a top element $\hat{1}$ and a bottom element $\hat{0}$. If $P$ is bounded, let $\bar{P}=P\setminus\Set{\hat{0},\hat{1}}$. Conversely, for any poset $P$, let $\widehat{P}=P\cup \Set{\hat{0},\hat{1}}$ where $\hat{0}$ and $\hat{1}$ are new elements adjoined so that $\hat{0}<x<\hat{1}$ for all $x\in P$.  A finite poset is said to be \Index{pure} if all maximal chains have the same length. And a poset is called \Index{graded} if it is finite, bounded and pure.

Given a finite poset $P$, its \Index{order complex} $\Delta(P)$ is the simplicial complex whose $k$-dimensional faces are the chains $x_0<x_1<\cdots<x_k$ of $P$.  A finite pure poset $P$ will be called \Index{(strongly) shellable} if its order complex $\Delta(P)$ is so. Note that a finite poset $P$ is (strongly) shellable if and only if $\widehat{P}$ is so.

Let $P$ be a finite pure poset of length $r-1$. Then $\widehat{P}$ is graded with a well-defined \Index{rank function} $\rho$, where $\rho(x)$ is defined to be
the common length of all unrefinable chains from  $\hat{0}$  to $x$ in $\widehat{P}$.  Obviously, $\rho(x) \in [r]$ for all $x \in P$. For any subset $S \subseteq [r]$, we define the \Index{rank-selected subposet} $P_S$ by  $\{x\in P \mid \rho(x)\in S\}$. Like \cite[Theorem 4.1]{MR570784}, we have

\begin{proposition}
    \label{Rank-Selection}
    If $P$ is a pure strongly shellable poset of length $r - 1$, then $P_S$ is strongly shellable for all $S\subseteq [r]$.
\end{proposition}

\begin{proof}
    The order complex $\Delta(P_S)=\Delta({P})[P_S]$. Obviously, if $\sigma$ is a maximal simplex in $\Delta({P})$, then $\sigma [P_S]$ is a maximal simplex in $\Delta(P_S)$. Now, we may apply Proposition \ref{nonpure restriction}.
\end{proof}

The proof for \cite[Proposition 4.2]{MR570784} also works for the following result:

\begin{proposition}
    If $P$ is a pure strongly shellable poset, then all intervals of $P$ are strongly shellable.
\end{proposition}

\begin{proof}
    The proof is standard. Assume that $P$ is strongly shellable and that $[x, y]$ is an interval of $P$. Let $\bdc:
    x_1 < x_2 < \cdots < x_g = x$ and $\bdd: y = y_1 < y_2 < \cdots < y_h$ be two unrefinable chains in $P$ such that $x_1$ is a minimal element and $y_h$ is a maximal element. Let $\bdm_1, \bdm_2,\dots, \bdm_t$ be the maximal chains in $P$ which contain $\bdc \cup \bdd$, and assume that they are listed in the order in which they appear in the strong shelling order $\succ$ of $P$.  Denote $(\bdm_i \setminus (\bdc \cup \bdd)) \cup \Set{x, y}$ by $\widetilde{\bdm}_i$.
    For each $i$ and $j$ with $1\le i< j \le t$, by the strong shellability of $\Delta(P)$, we can find $\bdm\succ \bdm_j$ with $\dis(\bdm,\bdm_j)=1$ and $\dis(\bdm,\bdm_i)=\dis(\bdm_i,\bdm_j)-1$ by Lemma \ref{d=d-1+1}.
    With the help of Lemma \ref{subset},
    we have $\bdc\cup \bdd \subseteq \bdm$. Thus, $\bdm=\bdm_k$ for some $k$. As $\bdm_k\succ \bdm_j$, we have $k<j$. Now, it is straightforward to verify that $\dis(\widetilde{\bdm}_k,\widetilde{\bdm}_j)=1$ and $\dis(\widetilde{\bdm}_k,\widetilde{\bdm}_i)=\dis(\widetilde{\bdm}_i,\widetilde{\bdm}_j)-1$.
    Therefore, $\widetilde{\bdm}_1,\cdots,\widetilde{\bdm}_t$ is a strong shelling order of the interval $[x, y]$.
\end{proof}

Let $P$ and $Q$ be two posets. The \Index{ordinal sum} $P\oplus Q$ is the poset on the disjoint union of $P$ and $Q$ defined by the rule: $x\le y$ in $P\oplus Q$ if and only if (i) $x,y\in P$ and $x\le y$ in $P$, or (ii) $x,y\in Q$ and $x\le y$ in $Q$, or (iii) $x\in P$ and $y\in Q$.  Like \cite[Theorem 4.4]{MR570784}, we have

\begin{proposition}
    \label{Ordinal-Sum}
    Let $P$ and $Q$ be two finite posets. Then the ordinal sum $P \oplus Q$ is strongly shellable if and only if both $P$ and $Q$ are strongly shellable.
\end{proposition}

\begin{proof}
    The order complexes satisfy $\Delta(P\oplus Q)=\Delta(P)*\Delta(Q)$. Thus, we can apply Proposition \ref{nonpure join}.
\end{proof}

However, other poset constructions, like direct product, cardinal power, interval poset, are easily seen to be not compatible with strong shellability.

\section{Relations with other shellable conditions}

In this section, we will show some relations among the concepts related to shellability. First, recall the following two conditions:

\begin{definition}
    \begin{enumerate}[a]
        \item A \Index{matroid complex} $\Delta$ is a simplicial complex with the exchange property: for any two distinct facets $F,G\in \calF(\Delta)$ and for any $i\in F\setminus G$, there exists some $j\in G\setminus F$ such that $(F\setminus\Set{i})\cup\Set{j}\in \Delta$.  Alternatively, $\Delta$ is called a \Index{matroid complex} if for every subset $W\subseteq \calV(\Delta)$, the induced subcomplex $\Delta_W$ is pure.  For other equivalent characterizations, see, for instance, \cite[Proposition III.3.1]{MR1453579}.
        \item A \Index{shifted} complex $\Delta$ on $[n]$ is a simplicial complex such that for each $F \in \Delta$, if $i \in F$ and $j \in [n]$ with $j > i$, then $(F \setminus \{i\}) \cup \{j\} \in \Delta$.
    \end{enumerate}
\end{definition}

\begin{example}
    Let $\Delta$ be the nonpure simplicial complex with the facet set $\calF(\Delta)$:
    \[
    \Set{ \{1, 2\}, \{1, 3, 6\}, \{1, 4, 6\}, \{1, 5, 6\}, \{2, 3, 6\}, \{2, 4, 6\}, \{2, 5, 6\}, \{3, 4, 5\}, \{3, 4, 6\}, \{3, 5, 6\}, \{4, 5, 6\} }.
    \]
    One can check directly that $\Delta$ is shifted. On the other hand, this simplicial complex is not strongly shellable; it suffices to compare the two facets $\{1,2\}$ and $\{3,4,5\}$.
\end{example}

Note that any matroid complex is pure, and as the above example shows, a nonpure complex needs not to be strongly shellable, even though it is shifted. Thus, in this section, we are mainly concerned with pure complexes.

\begin{proposition}
    \label{matroid implies ss}
    Matroid complexes are strongly shellable.
\end{proposition}

\begin{proof}
    Let $\Delta$ be a matroid complex defined on the set $[n]$. We will prove that $\Delta$ is strongly shellable by induction on $n$.
    Without loss of generality, we may assume that the vertex $n$ appears in some but not all facets of $\Delta$. Now,  $\Delta \setminus n$ and $\link_{\Delta}(n)$ are matroid complexes on the set $[n-1]$ with dimension $\dim(\Delta)$ and $\dim(\Delta)-1$ respectively. By inductive assumption, both $\Delta \setminus n$ and $\link_{\Delta}(n)$ are strongly shellable. Assume that $F_1, \cdots, F_s$ and $G_1, \cdots, G_t$ are strong shelling orders on $\mathcal{F}(\Delta \setminus n)$ and $\mathcal{F}(\link_{\Delta}(n))$ respectively. We claim that
    \[
    \succ: \quad F_1, \cdots, F_s, G_1 \cup \{n\}, \cdots, G_t \cup \{n\}
    \]
    is a strong shelling order on $\mathcal{F}(\Delta)$. It suffices to compare the pair $F_a\succ G_b\cup \{n\}$, where $1\le a \le s$ and $1\le b\le t$. Since $\Delta$ is a matroid complex and $n\in (G_b\cup\{n\})\setminus F_a$, we can find some $m\in F_a\setminus (G_b\cup\{n\})$ such that $( (G_b\cup\{n\})\setminus \{n\})\cup \{m \} = G_b\cup \{m\} \in \Delta$. Thus, $G_b\cup \{m\}\in \Delta\setminus n$ which means that $G_b\cup\{m\}=F_c$ for some $1\le c\le s$. Obviously, $F_c \succ G_b\cup\{n\}$.
\end{proof}

The proof of the above proposition also shows that, for a matroid complex, the reverse lexicographic order on the facet set gives a strong shelling order.

\begin{example}
    In \cite{arXiv:1311.0981}, the authors investigated the spanning tree complex of a connected graph.
    More generally, one can consider the spanning forests of a not necessarily connected graph $G$. A subset $F\subseteq \calE(G)$ is called a \Index{spanning forest} if for each connected component $L$ of $G$, $L \cap F$ is a spanning tree of $L$.
    We will write $s(G)$ for the set of spanning forests of $G$. The \Index{spanning forest complex} $\Delta_S(G)$ of $G$ will be the unique complex over $\calE(G)$, whose facet set is $s(G)$.
    It is not difficult to see that  $\Delta_S(G)$ is a matroid complex,
    which is known as the \Index{cycle matroid} of the graph $G$; see, for instance, \cite{MR2849819}.
    Indeed, if $G_1,\dots,G_s$ are the connected components of $G$, then
    \[
    \Delta_S(G)\cong \Delta_S(G_1)*\cdots * \Delta_S(G_s),
    \]
    the join of matroid complexes, hence again a matroid complex.
    In particular, by \ref{matroid implies ss}, $\Delta_S(G)$ is strongly shellable.

    To be more specific, if $|\calE(G)|=n$ and we label the edges of $G$ arbitrarily by distinct integers in $[n]$, the reverse lexicographic order on the sets of labels of the spanning forest gives rise to a strong shelling order.
    For instance, one can consider the simple graph $G$ of Figure \ref{Fig:Labeling edges-b} with the given labels. List all spanning trees by reverse lexicographic order on the labels:
    \[
    124, 134, 234, 125, 135, 235, 145, 245.
    \]
    This order is a strong shelling order on the facet set of the spanning forest complex $\Delta_S(G)$. The codimension one graph of $\Delta_S(G)$ is pictured in Figure \ref{Fig:Labeling edges-c}.
    \begin{figure}[!ht]
        \ffigbox[\FBwidth]{}{{%
        \begin{subfloatrow}[2]
            \ffigbox{\caption{A simple graph $G$ with labeled edges}}{
            \begin{minipage}[h]{\linewidth} \centering
                \includegraphics{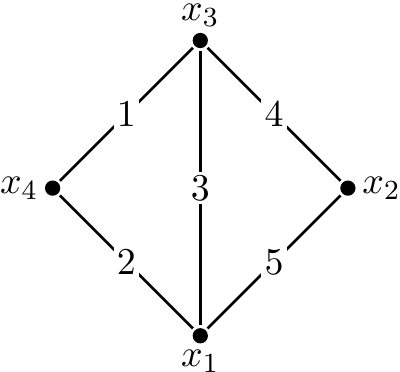}
            \end{minipage}
            \label{Fig:Labeling edges-b}
            }
            \ffigbox{\caption{$\Gamma(\Delta_S(G))$}}{
%
%
%
%
%
%
            \begin{minipage}[h]{\linewidth} \centering
                \includegraphics{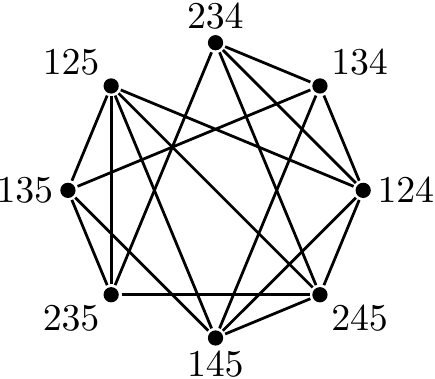}
            \end{minipage}
            \label{Fig:Labeling edges-c}
            }
        \end{subfloatrow}}
        \caption{Labels and the codimension one graph of the spanning complex} \label{Labeling edges} }
    \end{figure}
\end{example}

\begin{proposition}
    \label{strongly chain}
    Pure shifted complexes are strongly shellable.
\end{proposition}

\begin{proof}
    Let $\Delta$ be a pure shifted complex on the vertex set $[n]$. We will show that $\Delta$ is strongly shellable by induction on the cardinality of the vertex set. Assume that the result holds for the case of $n-1$.
    Without loss of generality, we may assume that the vertex $1$ appears in some but not all facets of $\Delta$.
    Note that $\Delta \setminus 1$ and $\link_{\Delta}(1)$ are pure shifted complexes on the vertex set $\{2, 3, \cdots, n\}$ of dimension $\dim(\Delta)$ and $\dim(\Delta)-1$ respectively. Thus, by inductive assumption, we can find strong shelling orders $F_1, \ldots, F_p$ and  $E_1, \ldots, E_q$ for $\mathcal{F}(\Delta \setminus 1)$ and $\mathcal{F}(\link_{\Delta}(1))$ respectively. We claim that
    \[
   \succ:\quad F_1, \ldots, F_p, E_1 \cup \{1\}, \ldots, E_q \cup \{1\}
    \]
    is a strong shelling order on $\mathcal{F}(\Delta)$. In fact, for $F_i\succ E_j \cup \{1\}$ in $\mathcal{F}(\Delta)$, there exists an integer $t \in F_i \setminus (E_j \cup \{1\})$. Since $\Delta$ is shifted, $E_j \cup \{t\} \in \mathcal{F}(\Delta)$.
    Obviously, $E_j\cup \{t\}\in \calF(\Delta\setminus 1)$ which means $E_j\cup\{t\}=F_k$ for some $1\le k \le p$. Of course, $F_k\succ E_j\cup \{1\}$.
    It is also easy to see that $\dis(E_j \cup \{t\}, E_j \cup \{1\})=1$ and $\dis(F_i, E_j \cup \{t\})=\dis(F_i, E_j \cup \{1\})-1$. This completes the proof.
\end{proof}

The proof of the above proposition also shows that, for a pure shifted complex,  if $L_1, \cdots, L_s$ is a linear order on the facet set given by the lexicographic order, then $L_s, \cdots, L_1$ is a strong shelling order.

Recall that a $d$-dimensional pure complex $\Delta$ is \Index{vertex decomposable} if either $\Delta$ is a $d$-simplex, or there exists a vertex $x$ of $\Delta$ (called a \Index{shedding vertex}), so that
\begin{enumerate}[i]
    \item $\Delta \setminus x$ is pure of dimension $d$ and vertex decomposable, and
    \item $\link_{\Delta}(x)$ is pure of dimension $d-1$ and vertex decomposable.
\end{enumerate}
This definition was introduced by Provan and Billera \cite[Definition 2.1]{MR0593648}. For the nonpure version, see \cite[Definition 11.1]{MR1401765}.
It is not difficult to see that matroid complexes are pure vertex decomposable, see \cite[Theorem 3.2.1]{MR0593648}.
It was also shown in \cite{MR1401765} that for nonpure complexes, we have
\[
\text{shifted $\Longrightarrow$ vertex decomposable $\Longrightarrow$ shellable}.
\]

Recall that Kokubo and Hibi \cite{MR2260118} introduced the weakly polymatroidal ideals. We will consider the Eagon complexes of squarefree weakly polymatroidal ideals. To be more precise,

\begin{definition}
    \label{WM}
    A simplicial complex $\Delta$ on the vertex set $[n]$ is called a \Index{weakly matroid complex}, if for each distinct facets $G$ and $F$, with respect to the unique vertex $q\in G\setminus F$ such that for each $i<q$, $i\in F$ if and only if $i\in G$, there exists some $p\notin G$ with $q<p\le n$ such that $(\{p\}\cup G) \setminus \{q\}\in \Delta$.
\end{definition}

It is immediate that shifted complexes and matroid complexes are weakly matroidal. On the other hand, by \cite[Theorem 2.5]{MR2845598}, weakly matroid complexes are vertex decomposable.

In \cite{MR2771603}, Hachimori and Kashiwabara introduced hereditary-shellable complexes. According to them, a complex is called \Index{hereditary-shellable} if all its restrictions are (nonpure) shellable. For pure complexes, they \cite{hachimori2009hereditary} established the relations
\[
\xymatrixrowsep{0pc}
\xymatrixcolsep{2pc}
\xymatrix{
&\text{hereditary-shellable} \ar@{=>}[dr] & \\
\text{matroid} \ar@{=>}[ur] \ar@{=>}[dr] & & \text{shellable}. \\
&\text{vertex decomposable} \ar@{=>}[ur]&
}
\]
It is easy to find a pure vertex decomposable complex which is not hereditary-shellable. On the other hand,
\cite[Example 4.6]{hachimori2009hereditary} shows a pure hereditary-shellable complex which is not vertex decomposable.

In the following, we will give additional examples which show that there is no implication between strong  shellability with  hereditary-shellability or vertex decomposability.

\begin{example}
    [Strongly shellable, but not vertex decomposable]
    We have already mentioned that \cite[Example 4.6]{hachimori2009hereditary} gives a pure complex which is hereditary-shellable but not vertex decomposable. Actually, one can check directly that it is also strongly shellable. The facet set of this complex has the following strong shelling order:
    \begin{center}
        \begin{minipage}[h]{0.75\linewidth}
            $\{a_1, b_1, c_1\}$,
            $\{a_1, b_1, c_2\}$,
            $\{a_1, b_1, c_3\}$,
            $\{a_1, b_1, c_4\}$,
            $\{a_1, b_2, c_1\}$,
            $\{a_1, b_2, c_2\}$,
            $\{a_1, b_2, c_3\}$,
            $\{a_1, b_2, c_4\}$,
            $\{a_1, b_3, c_1\}$,
            $\{a_1, b_3, c_2\}$,
            $\{a_1, b_3, c_3\}$,
            $\{a_1, b_3, c_4\}$,
            $\{a_1, b_4, c_1\}$,
            $\{a_1, b_4, c_2\}$,
            $\{a_1, b_4, c_3\}$,
            $\{a_1, b_4, c_4\}$,
            $\{a_2, b_1, c_1\}$,
            $\{a_2, b_1, c_2\}$,
            $\{a_2, b_1, c_3\}$,
            $\{a_2, b_1, c_4\}$,
            $\{a_2, b_2, c_1\}$,
            $\{a_2, b_2, c_2\}$,
            $\{a_2, b_2, c_3\}$,
            $\{a_2, b_2, c_4\}$,
            $\{a_2, b_3, c_1\}$,
            $\{a_2, b_3, c_2\}$,
            $\{a_2, b_3, c_3\}$,
            $\{a_2, b_3, c_4\}$,
            $\{a_2, b_4, c_1\}$,
            $\{a_2, b_4, c_2\}$,
            $\{a_2, b_4, c_3\}$,
            $\{a_2, b_4, c_4\}$,
            $\{a_3, b_1, c_1\}$,
            $\{a_3, b_1, c_2\}$,
            $\{a_3, b_1, c_3\}$,
            $\{a_3, b_1, c_4\}$,
            $\{a_3, b_2, c_1\}$,
            $\{a_3, b_2, c_2\}$,
            $\{a_3, b_2, c_3\}$,
            $\{a_3, b_2, c_4\}$,
            $\{a_3, b_3, c_1\}$,
            $\{a_3, b_3, c_2\}$,
            $\{a_3, b_3, c_3\}$,
            $\{a_3, b_3, c_4\}$,
            $\{a_3, b_4, c_1\}$,
            $\{a_3, b_4, c_2\}$,
            $\{a_3, b_4, c_3\}$,
            $\{a_3, b_4, c_4\}$,
            $\{a_4, b_1, c_1\}$,
            $\{a_4, b_1, c_2\}$,
            $\{a_4, b_1, c_3\}$,
            $\{a_4, b_1, c_4\}$,
            $\{a_4, b_2, c_1\}$,
            $\{a_4, b_2, c_2\}$,
            $\{a_4, b_2, c_3\}$,
            $\{a_4, b_2, c_4\}$,
            $\{a_4, b_3, c_1\}$,
            $\{a_4, b_3, c_2\}$,
            $\{a_4, b_3, c_3\}$,
            $\{a_4, b_3, c_4\}$,
            $\{a_4, b_4, c_1\}$,
            $\{a_4, b_4, c_2\}$,
            $\{a_4, b_4, c_3\}$,
            $\{a_4, b_4, c_4\}$,
            $\{a_2, b_1, b_2\}$,
            $\{a_2, a_3, b_2\}$,
            $\{a_3, b_2, b_3\}$,
            $\{a_3, a_4, b_3\}$,
            $\{b_1, c_1, c_2\}$,
            $\{b_1, b_4, c_2\}$,
            $\{b_4, c_2, c_3\}$,
            $\{b_4, b_2, c_3\}$,
            $\{c_1, a_3, a_1\}$,
            $\{c_1, c_4, a_1\}$,
            $\{c_4, a_1, a_4\}$,
            $\{c_4, c_2, a_4\}$.
        \end{minipage}
    \end{center}
    Using the language in \cite[Example 4.6]{hachimori2009hereditary}, the above strong shelling order first starts with the matroidal part, and then proceeds by carefully arranging the 12 extra facets.
\end{example}

\begin{example}
    [Strongly shellable, but not hereditary-shellable]
    Let $\Delta$ be a pure complex with the facet set:
    \[
    \Set{ \{1, 2, 8\}, \{1, 2, 5\}, \{2, 5, 6\}, \{1, 2, 4\}, \{1, 2, 7\}, \{1, 3, 4\}, \{1, 2, 6\}, \{1, 2, 3\} }.
    \]
    It is direct to check that $\Delta$ is strongly shellable with respect to the above given order. On the other hand, the restriction of $\Delta$ to the subset $W=\Set{3,4,5,6,7,8}\subseteq \calV(\Delta)$ has facets $\{3,4\},\{5,6\},\{7\},\{8\}$. Obviously, $\Delta_W$ is not shellable. Therefore, the original complex $\Delta$ is not hereditary-shellable.
\end{example}

\begin{example}
    [Hereditary-shellable, weakly matroid, but not strongly shellable]
    Let $\Delta$ be a pure complex with the facet set:
    \[
    \Set{ \{1, 2, 6\}, \{1, 3, 4\}, \{1, 4, 6\}, \{2, 3, 5\}, \{2, 5, 6\}, \{3, 4, 5\}, \{3, 4, 6\}, \{3, 5, 6\} }.
    \]
    It is direct to check that $\Delta$ is hereditary-shellable and weakly matroid, but not strongly shellable.
\end{example}

\begin{proposition}
    \label{hereditary chain}
    Shifted complexes are hereditary-shellable.
\end{proposition}

\begin{proof}
    It suffices to note that for a shifted complex, any restriction is also shifted.
\end{proof}

Combining the concepts of hereditary-shellable and strongly shellable, we have the following natural definition.

\begin{definition}
    A complex is called \Index{hereditarily strongly shellable} if all its restrictions are strongly shellable.
\end{definition}

Clearly, we have the following implications:
\[
\xymatrixrowsep{0pc}
\xymatrixcolsep{1.5pc}
\xymatrix{
& &   \text{hereditary-shellable}\\
\text{matroid} \ar@{=>}[r]& \text{    hereditarily strongly shellable} \ar@{=>}[ur] \ar@{=>}[dr]&  \\
& & \text{strongly shellable}
}
\]
The implications are strict.

\begin{example}
    Let $\Delta$ be the pure complex with facets:
    \[
    \{1, 2, 3\}, \{1, 2, 4\}, \{1, 2, 5\}, \{1, 2, 6\}, \{1, 3, 4\}, \{1, 3, 5\}, \{1, 3, 6\}, \{2, 4, 5\}.
    \]
    One can check that $\Delta$ is hereditarily strongly shellable.
    On the other hand, the restriction of $\Delta$ to the subset $W=\Set{4,5,6}\subseteq \calV(\Delta)$ has facets $\{4,5\}$ and $\{6\}$. Therefore, the induced complex $\Delta_W$ is not pure and the original complex $\Delta$ is not a matroid complex.
\end{example}

The implications among pure shellable complexes that we encountered in this section are summarized in Figure \ref{Fig:relation}.

\begin{figure}[!ht]
    \xymatrix{
    \text{shifted}  \ar@{=>}[ddr]\ar@{=>}[ddrr]\ar@{=>}[d] &  & \text{matroid} \ar@{=>}[d]\ar@{=>}[dll] \\
    \text{weakly matroid} \ar@{=>}[d] &  & \text{hereditarily strongly shellable} \ar@{=>}[d] \ar@{=>}[dl] \\
    \text{vertex decomposable} \ar@{=>}[dr]& \text{hereditary-shellable} \ar@{=>}[d] & \text{strongly shellable}  \ar@{=>}[dl]\\
    & \text{shellable} &   
    }
    \caption{Relations among pure shellable complexes} \label{Fig:relation}
\end{figure}
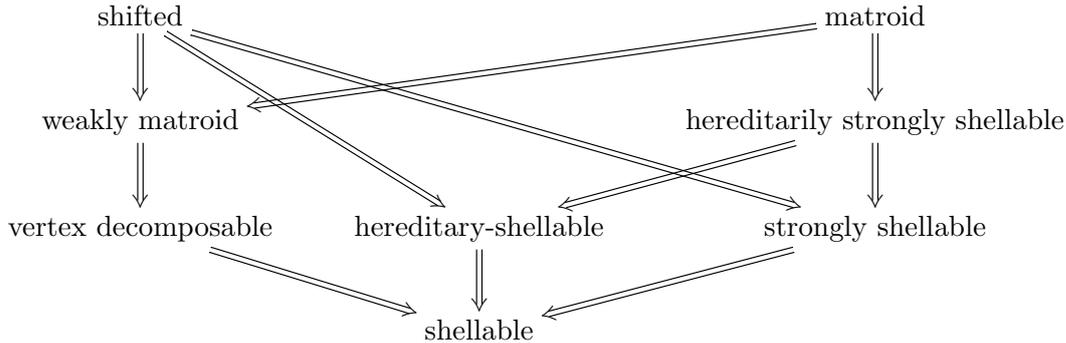

\section{Deciding strong shellability of pure complexes with $h$-assignments}

In this final section, we briefly talk about how to decide the strong shellability of a pure complex. In \cite{ISI:000292480700004} Moriyama considered a similar question for deciding the shellability of pure complexes.

Let $\Delta$ be a $d$-dimensional pure complex with $h$-vector
\[
\bdh(\Delta)=(h_0(\Delta),\dots,h_{d+1}(\Delta))\in \NN^{d+2}.
\]
In $\Delta$, a $(d-1)$-dimensional face is called a \Index{ridge} of $\Delta$. A ridge contained in only one facet will be called a \Index{boundary ridge}.

An \Index{$h$-assignment} $A$ of $\Delta$ is a assignment $A:\calF(\Delta)\to [d+1]$ such that $|A^{-1}(i)|=h_i(\Delta)$ for each $i$.  With respect to this $h$-assignment, a facet $F$ of $\Delta$ is called a \Index{candidate facet} if $F$ contains exactly $d+1-A(F)$ boundary ridges of $\Delta$.

Given a candidate facet $F$ of $\Delta$ with respect to the $h$-assignment, we can apply a \Index{removing step} by
\begin{enumerate}[i]
    \item replacing $\Delta$ by $\Delta'=\Braket{G\in \calF(\Delta):G\ne F}$, and
    \item replacing $A$ by the its restriction on $\calF(\Delta')$.
\end{enumerate}

\begin{theorem}
    [{\cite[Theorem 1.3]{ISI:000292480700004}}]
    \label{Moriyama}
    A pure complex $\Delta$ is shellable if and only if there exists an $h$-assignment such that we can remove all facets of $\Delta$ by applying the removing steps successively.
\end{theorem}

From its proof, we know that given a shelling order of $\Delta$, one has a natural $h$-assignment in which the last facet of the shelling will be a candidate facet. Conversely, given an $h$-assignment which allows the removing steps above, we can build a shelling order by reversing the removal order.

Given an $h$-assignment, a candidate facet $F$ is called a \Index{strong candidate facet} if for each $G\ne F$, there exists a facet $H$ such that $\dis(F,H)=1$ and $\dis(G,H)=\dis(G,F)-1$.

It follows immediately from Lemma \ref{d=d-1+1} and Theorem \ref{Moriyama} that

\begin{theorem}
    A pure complex $\Delta$ is strongly shellable if and only if there exists an $h$-assignment such that we can remove all facets of $\Delta$ by applying the removing steps successively such that we remove a strong candidate facet at each step.
\end{theorem}

The benefit of applying this deciding method is the same as that for Moriyama's suggestion, namely, instead of checking all $|\calF(\Delta)|!$ possible cases by definition, we only need to check, roughly speaking,
\[
\frac{|\calF(\Delta)|!}
{h_0(\Delta)!\cdots h_{d+1}(\Delta)!}
\]
possible cases. The last integer is the number of  all $h$-assignments of $\Delta$.


\end{document}